\documentclass{amsart}
\usepackage{amsmath}

\usepackage{amsfonts}
\usepackage{amssymb}
\usepackage{graphicx}%
\newtheorem{mytheo}{Theorem}[section]
\newtheorem{mydef}[mytheo]{Definition}
\newtheorem{myrem}[mytheo]{Remark}
\newtheorem{mylem}[mytheo]{Lemma}
\newtheorem{mycoro}[mytheo]{Corollary}
\newtheorem{mypropo}[mytheo]{Proposition}

\newtheorem{mycon}[mytheo]{Conjecture}
\usepackage{fancyhdr}
\usepackage{enumerate}
\usepackage[dvipsnames]{xcolor}  % 支持更多命名颜色（如 RoyalBlue）
\usepackage[colorlinks=true,
            linkcolor=RoyalBlue,
            citecolor=RoyalBlue,
            urlcolor=RoyalBlue]{hyperref}

\theoremstyle{remark}

\numberwithin{equation}{section}

 \allowdisplaybreaks[4]
\begin{document}

\title{Singularity of biased discrete random matrices}

\pagestyle{fancy}

%清除原页眉页脚样式
\fancyhf{} 

%R：页面右边；O：奇数页；\leftmark：表示“一级标题”

%C：页面中间
\fancyhead[CO]{\footnotesize DISCRETE RANDOM MATRIX }

\fancyhead[CE]{\footnotesize Z. SONG }

\fancyhead[LE]{\thepage}
\fancyhead[RO]{\thepage}
\renewcommand{\headrulewidth}{0mm}
%    Information for first author
%    Information for second author
%    Information for second author

\author{Zeyan Song}
\address{Shandong University, Jinan, 250100, China.}
\email{zeyansong8@gmail.com}
\subjclass[2020]{60B20, 15B52}

\date{}

\keywords{random matrices}

\begin{abstract}
We study the singularity probability of $n \times n$ random matrices with i.i.d. entries from highly biased discrete distributions. We obtain sharp non-asymptotic bounds for this probability and derive estimates on the least singular values. Our method combines combinatorial, geometric, and probabilistic techniques such as sphere decomposition and anticoncentration inequalities. The results extend classical invertibility theory to biased discrete settings and resolve an open problem by characterizing the dominant causes of singularity in biased discrete random matrices, namely the presence of zero columns or linearly dependent column pairs.
\end{abstract}

\maketitle

\section{Introduction}\label{Intro}
Let $A=(a_{ij})_{n\times n}$ be an $  n\times n$ random matrix with independent and identically distributed entries. The singularity probability of $A$ is a classical problem in the literature. An important and widely studied question is: What is the probability that a matrix $ B$  with entries uniformly distributed in $\{ -1,1 \}$ is singular, for which the celebrated conjecture states:
\begin{align}
    \textsf{P}\left( \det(B)=0 \right)=(1+o_{n}(1))2n^{2}2^{-n}.\nonumber
\end{align}
Note that the right-hand side is equal to the probability that the matrix $B$ contains two columns or two rows that are identical or negative to each other.

The singularity probability was first investigated by Koml\'{o}s in the 1960s, who showed that it is equal to $O(n^{-1/2})$. Thirty years later, this upper bound was improved by Kahn, Koml\'{o}s, and Szemer\'{e}di in \cite{KKS_jams} to $\left( 0.998+o_{n}(1) \right)^{n}$, which is the first exponential bound. Regarding the exponential bound, Tao and Vu reduced the upper bound to $\left( 3/4+o_{n}(1) \right)^{n}$ in \cite{Tao_rsa, Tao_jams}. Furthermore, Bourgain, Vu, and Wood provided that the upper bound is $(2^{-1/2}+o_{n}(1))^{n}$.

A landmark contribution to the study of singularity probability is the work of Rudelson and Vershynin in \cite{Rudelson_adv}, who established precise bounds on the smallest singular value of subgaussian random matrices. Specifically, they obtained for all $\varepsilon \ge 0$,
\begin{align}
    \textsf{P}\left( s_{\min}(A) \le \varepsilon n^{-1/2} \right) \le C\varepsilon + e^{-cn}. \nonumber
\end{align}

In fact, their proposed geometric approach links combinatorial random matrices with asymptotic geometric analysis. In addition, they introduced the LCD as a primary tool for studying the Littlewood–Offord problem. All of these have become among the important methods in this field.

Tikhomirov obtained the sharpest result currently conjectured in \cite{Tikhomirov}, where it is proved that
\begin{align}
    \textsf{P}\left( B \text{ is singular} \right) =(1/2+o_{n}(1))^{n}.\nonumber
\end{align}

He innovatively proposed the technique of ``Inversion of randomness", thus completing the proof. In fact, Tikhomirov proved that the above conclusion also holds for random matrices whose entries are uniformly distributed on \(\{0,1\}\). This is another important model of interest, the Bernoulli random matrix. Its distinction from the Rademacher random matrix discussed above lies in the need to account for the probability of having all zero rows or columns, and the probability that two columns are linearly dependent is much smaller than the probability that a single column is zero. More precisely, if we define a Bernoulli$(p)$ random matrix $B_{p}$ as a matrix whose entries are i.i.d., taking values $0$ with probability $1-p$ and $1$ with probability $p$, then a natural conjecture is that 
\begin{align}\label{conjecture Bp}
    \textsf{P}\left( B_{p} \text{ is singular} \right)=(2+o_{n}(1))n(1-p)^{n}.
\end{align}

If $p$ is a fixed constant, Tikhomirov has already established a bound of $(1 - p + o_{n}(1))^{n}$. The key issue then is how to obtain a sharp estimate for the $o_{n}(1)$ term. Jain, Sah, and Sawhney \cite{JSS_gafa} improved Tikhomirov’s argument of ``inversion of randomness" and consequently proved that \eqref{conjecture Bp} holds when $p <1/2$ is a fixed constant.

When $p = o_{n}(1)$, the singularity probability remains an important and interesting problem. Note that a Bernoulli random matrix is the adjacency matrix of a directed graph, so this regime is closely related to sparse graph models. However, since most previous work focused on the case where $p$ is a fixed constant, there is a lack of suitable analytical tools for this setting. The main difficulty arises from the fact that when $p$ becomes too small, the traditional $\varepsilon$-net approach does not work for sparse vectors.

In this setting, Basak and Rudelson \cite{BR_adv} established precise bounds on the smallest singular value of $B_{p}$, showing that for $pn$ above a logarithmic threshold, the matrix remains invertible with high probability. In particular, there exists $C \ge 2$ such that for all $C\log n \le pn \le C^{-1}n$ and $\varepsilon \ge 0$:
\begin{align*}
    \textsf{P}\left( s_{\min}(B_{p}) \le \varepsilon \exp\left( -C\log(1/p)/\log(np) \right) \sqrt{p/n} \right) \le C\varepsilon+e^{-cpn}.
\end{align*}

Of course, their proof relies on the LCD and follows the classical framework in \cite{Rudelson_adv}, yielding an exponential upper bound. To obtain sharper estimates, Litvak and Tikhomirov \cite{LT_duke} introduced the ``U-degree" and incorporated techniques from the study of $d$-regular random matrices in the analysis of sparse Bernoulli random matrices. As a result, they showed that \eqref{conjecture Bp} holds in the regime $C \log n/n \le p \le C^{-1}$. For works on $d$-regular random matrices, the readers may consult \cite{LLTTY_ptrf}. 

The final range is $C \log n / n \ge p \ge \log n / n$. The difficulty here lies in the fact that Litvak and Tikhomirov’s treatment of “sparse vectors’’ breaks down in this regime. By modifying the analysis in this part, Huang \cite{Huang_rank} ultimately completed the proof for the remaining range.

Furthermore, a natural extension is to consider random matrices with general discrete distributions. Instead of being restricted to $\{0,1\}$ or $\{\pm 1\}$, the entries can take values from a finite set with arbitrary probabilities. For discrete random matrices, since we cannot directly determine the relative magnitudes of the probabilities that a column is zero and that two columns are linearly dependent, a natural conjecture is that the singularity probability should be jointly governed by these two events. Another viewpoint on singularity probability was given by Bourgain, Wood, and Vu \cite{Bour_jfa}. Let $\bar{p}$ denote the maximum probability of atoms of the underlying discrete random variable. They proved an upper bound of $(\sqrt{\bar p} + o(1))^{n}$ for the singularity probability of discrete random matrices, and further conjectured that the correct upper bound should be $(\bar{p}+o(1))^{n}$.

In the case of Bernoulli random matrices, this bound is essentially optimal, as $\bar p = 1 - p$. For more general distributions, it is not optimal. For example, when the random variable is equal to $1$ with probability $1-p$ and $-1$ with probability $p$—the probability that two columns of the random matrix are linearly dependent is $\left((1-p)^{2} + p^{2} + o(1)\right)^{n}$, which is clearly much smaller than $(1 - p + o(1))^{n}$. In fact, we may formulate the following conjecture:
\begin{mycon}\label{Conjecture A}
    Let $x$ be a discrete random variable and let $x'$ be an independent copy of $x$. Let $A_{n}(x)$ be an $n\times n$ random matrix with entries independent and identically distributed according to $x$. Then  
\begin{align*}
      \textsf{P}\left( A_{n}(x) \text{ is singular}  \right) & =(2+o_{n}(1))\textsf{P}\left( x=0 \right)^{n} \nonumber \\ 
    &\quad +\left( 1+o_{n}(1) \right)n(n-1)\left( \textsf{P}\left( x'=x \right)^{n}+\textsf{P}\left( x'=-x \right)^{n}  \right). 
\end{align*}
\end{mycon}

More recently, Jain, Sah and Sawhney \cite{JSS_gafa} resolved the conjecture in the case where the distribution is not uniform in support and $\sup_{r \in \mathbb{R}}\textsf{P}\left( x=r \right)$ is a fix constant. Their work confirmed that the singularity probability is indeed governed by the most likely causes of degeneracy, such as zero rows/columns or linearly dependent pairs.

The natural question that arises is whether Conjecture \ref{Conjecture A} remains valid when
$\sup_{r \in \mathbb{R}} \mathsf{P}(x = r) = 1 - o(1)$. This is also the main focus of the present paper: we prove that under this condition, Conjecture \ref{Conjecture A} indeed holds. Before presenting our main theorem, we introduce the following assumptions. Without loss of generality, set 
\begin{align}\label{Standardize for discrete r.v.}
    \eta = \delta \xi + b, 
\end{align}
where $b \in \mathbb{R}$, $\delta$ is a Bernoulli($p$) random variable, and $\xi$ is a discrete random variable with support
\begin{align}\label{Condition}
   S_{\xi}:=\{ a_{1},\dots,a_{L}  \}, \quad |a_{i}| \le 1, \quad \textsf{P}\!\left( \xi=a_{i} \right)=p_{i}, \quad p_{1}+\dots+p_{L}=1,
\end{align}
so that $S_{\xi}$ denotes the support set of $\xi$.

We are now in a position to state the main theorem of this paper, which settles the invertibility problem for random matrices with highly skewed discrete entries.

\begin{mytheo}\label{Theorem A}
Let $\eta$ be a discrete random variable and set 
\[
p = 1-\sup_{r \in \mathbb{R}} \textsf{P}(\eta = r).
\] 
Suppose $\eta$ is standardized as in \eqref{Standardize for discrete r.v.}. Then there exist constants $C_{\ref{Theorem A}}>1$ and $n_{\ref{Theorem A}} \in \mathbb{N}$, depending only on $\xi$, such that the following holds. For all $n \ge n_{\ref{Theorem A}}$ and
\[
\frac{C_{\ref{Theorem A}}\log n}{n} \le p \le C_{\ref{Theorem A}}^{-1},
\]
let $M_{n}$ be an $n \times n$ random matrix whose entries are i.i.d. copies of $\eta$. Then
\begin{align}
    \textsf{P}\left( M_{n} \text{ is singular} \right) 
    &= (2+o_{n}(1))n\textsf{P}\left( \eta =0 \right)^{n} \nonumber \\
    &\quad +(1+o_{n}(1))n(n-1)\textsf{P}\left( \eta'=\eta \right)^{n}, \nonumber
\end{align}
where $\eta'$ is an independent copy of $\eta$.

Moreover, for all $t>0$, we have
\begin{align}
    \textsf{P}\!\left( s_{\min}(M_{n}) \le t\exp(-3\log^{2}(2n)) \right) 
    \le t + \textsf{P}\!\left( M_{n} \text{ is singular} \right). \nonumber
\end{align}

\end{mytheo}
\begin{myrem}
Combining Theorem~\ref{Theorem A} with the work of Jain, Sah, and Sawhney \cite{JSS_gafa}, 
the conjecture on the singularity probability of discrete random matrices is now completely resolved for $x$ is a bounded discrete distribution (with bound independent of $n$) and the mass concentrated at a single point lies in the regime
\[
\frac{1}{2}> 1-\sup_{r\in \mathbb{R}}\textsf{P}(x=r) > \frac{C\log n}{n}.
\]

Furthermore, the asymptotic bound in Theorem~\ref{Theorem A} coincides exactly with the probability of the most likely singularity of $M_{n}$: 
the first term arises from the presence of a zero row or column, and the second term from the event that two rows or two columns are equal or opposite. 
Hence, the singularity probability is completely explained by these natural events.

\end{myrem}

We now briefly describe the main difficulties of the problem and our approach to resolving them. The first major obstacle is that the existing partition of the vector space breaks down when attempting to obtain the probability bounds we need. This requires us to redesign the partition of the vector space and to establish new probability estimates for this refined decomposition. 

The second major difficulty is that the notion of ``U-degree" is effective only for Bernoulli random variables. To address more general discrete distributions, we introduce a broader characterization, which we call the ``Randomized U-degree", to capture their structural properties. More specifically, we provide a method that, on the one hand, estimates the anticoncentration inequalities for discrete random variables, and on the other hand, establishes that the RUD enjoys desirable properties previously possessed by notions such as the LCD and the U-degree. This indicates that the RUD should have broader applications in the study of sparse random matrices.

\textbf{Organization of this paper}
In Section \ref{Notation and Proof sketch}, we present the notation and outline of the proof used in this paper. Section \ref{Preliminaries} provides the preliminary knowledge. Then, in Sections \ref{Unstructured vectors section} and \ref{Structured vectors section}, we analyze the probability of the existence of two types of vector in linear spaces, respectively. Finally, in Section \ref{Proof of main result section}, we establish the proofs of the two main theorems of this paper. 
\section{Notation and Proof sketch}\label{Notation and Proof sketch}
\subsection{Notation}
We denote by $[n]$ the set of natural numbers from $1$ to $n$. Given a vector $x \in \mathbb{R}^{n}$, we denote by $\Vert x\Vert_{2}$ its standard Euclidean norm: $\Vert x\Vert_{2}=\left(\sum_{j\in [n]}{x_{j}^{2}} \right)^{\frac{1}{2}}$, the supnorm is denoted $\Vert x\Vert_{\infty}=\max_{i}{|x_{i}|}$. Fixing $\textbf{1}_{n}:=\left(1,\dots,1  \right) \in \mathbb{R}^{n}$ and $\textbf{e}:=\frac{1}{\sqrt{n}}\textbf{1}_{n}$. Let $P_{\textbf{e}}$ be the projection on $\textbf{e}^{\perp}$, and let $P_{\textbf{e}^{\perp}}$ be the projection on $\textbf{e}$. Furthermore, consider the norm on $\mathbb{R}^{n}$ defined by 
\begin{align}
    \Vert x\Vert_{\textbf{e}}^{2}=\Vert P_{\textbf{e}}x\Vert_{2}^{2}+pn\Vert P_{\textbf{e}^{\perp}}x\Vert_{2}^{2}.\nonumber
\end{align}

For the set $I \subset [n]$, we define $x_{I}$ to be the vector composed of all the coordinates of $x$ whose indices belong to $I$. Let $x^{*}$ denote a nonincreasing rearrangement of the absolute values of the components of a vector $x$ and the permutation $\sigma_{x}$ satisfying $|x_{\sigma_{x}(j)}|=x^{*}_{j}$.

The unit sphere of $\mathbb{R}^{n}$ is denoted by $S^{n-1}$. The cardinality of a finite set $\mathrm{I}$ is denoted by $\left| \mathrm{I} \right|$. Define the L\'{e}vy function of a random vector $\xi \in \mathbb{R}^{n}$ and $t>0 $ as 
\begin{align}
    \mathcal{L}\left( \xi,t \right):=\sup_{w \in \mathbb{R}^{n}}{\textsf{P}\left( \Vert \xi -w \Vert_{2} \le t \right)}.\nonumber
\end{align}

Let $H$ be an $m \times n$ matrix, define $\mathrm{R}_{j}(H)$($\mathrm{C}_{j}(H)$) as the $j$-th rows(columns) of $H$. Define the spectral norm of $H$ by $\Vert H\Vert := \sup_{\Vert x\Vert_{2}=1}{\Vert Hx\Vert_{2}}$, the least singular value of $H$ by $s_{\min}(H)$, and ``Hilbert-Schmidt'' norm by $\Vert H\Vert_{\mathrm{HS}}:=\left( \sum_{i,j}{h_{ij}^{2}} \right)^{1/2}$.

In this paper, we define $c$, $c',\dots$ as some fixed constant and define $c\left( u\right)$, $C\left( u\right)$ as a constant related to $u$, and they depend only on the parameter $u$. Their value can change from line to line.
\subsection{Proof sketch}
In this subsection, we present an overview of the proof and our principal innovations. First, we adopt standard approaches in the field, such as sphere decomposition, net arguments, and anticoncentration inequality. For the lower bound, earlier work \cite{LT_duke, JSS_gafa} shows that the singularity probability is at least the RHS of Theorem \ref{Theorem A}.

For the upper bound, define $\textsf{P}_{singular}$ as 
\begin{align}\label{Singularity probability}
    \textsf{P}_{singular}:= 2n\textsf{P}\left( \eta =0 \right)^{n} + n(n-1)\textsf{P}\left( \eta'=\eta \right)^{n}.
\end{align}

Let us begin with the decomposition of $\mathbb{R}^{n}$, following the framework of \cite{LT_duke}, we divide $\mathbb{R}^{n}$ into the unstructured part $\mathcal{V}_{n}$ and its complement. By the natural and standard approach, it is sufficient to prove that 
\begin{align}
    \textsf{P}\left( \{ M_{n}x=0 \text{ for some }  x \in \mathcal{V}_{n}\} \cap \{ M_{n}x \ne 0 \text{ for all } x \notin \mathcal{V}_{n}\}   \right)=o_{n}(\textsf{P}_{singular})\nonumber
\end{align}
and 
\begin{align}
    \textsf{P}\left(  M_{n}x=0 \text{ for some }  x \notin \mathcal{V}_{n}  \right)=(1+o_{n}(1))\textsf{P}_{singular}/2.\nonumber
\end{align}

For the complement of the unstructured vectors, our lemma in Section \ref{Decomposition of Rn} shows that it is contained in the union of ``Steep part'' $\mathcal{T}$  and ``Spread part'' $\mathcal{R}$.

Otherwise, when $b = 0$, the first term on the right-hand side of the inequality in Theorem \ref{Theorem A} dominates, and the second term is strictly of smaller order.  When $b \ne 0$, the second term becomes the dominant one.  

In either situation, for every $R \ge  2$ and any $p:=1-\sup_{r \in \mathbb{R}}\textsf{P}\left( \eta = r  \right)$ with $p \ge \frac{\log{n} }{n}$, we have
\begin{align}
    \exp(-Rpn):=o_{n}(\textsf{P}_{singular}).\nonumber
\end{align}
 
 Thus, combining the two above results, we only need to prove that for some $R \ge 2$
 \begin{align}\label{Unstructured probability}
     \textsf{P}\left( \{ M_{n}x=0 \text{ for some }  x \in \mathcal{V}_{n}\} \cap \{ M_{n}x \ne 0 \text{ for all } x \notin \mathcal{V}_{n}\}   \right) \le e^{-Rpn}
 \end{align}
 \begin{align}\label{Spread probability}
     \textsf{P}\left( M_{n}x =0 \text{ for some } x \in \mathcal{R} \right) \le e^{-Rpn}
 \end{align}
 and
 \begin{align}\label{Steep probability}
     \textsf{P}\left( M_{n}x= 0 \text{ for some } x \in \mathcal{T}  \right)= (1+o_{n}(1))\textsf{P}_{singular}/2.
 \end{align}

Let us now turn to the unstructured vector component. Following earlier practice, anticoncentration must again be taken into account. One of the first, most influential, and now standard techniques for this purpose is the ``LCD'' introduced by Rudelson and Vershynin in \cite{Rudelson_adv}, which has become a cornerstone of the field; for further characterizations and applications of the LCD, the reader is referred to \cite{Campos_jams, Campos_pi, Livshyts_jam, Livshyts_aop}. 

However, an inevitable drawback of using the LCD is that it only yields exponential upper bounds and cannot provide optimal estimates. To address this, Litvak and Tikhomirov \cite{LT_duke} introduced the ``U-degree" for Bernoulli random matrices. However, the key limitation is that the U-degree is effective only for Bernoulli random variables.

To characterize general discrete random variables, we introduce the ``Randomized Unstructured Degree (RUD)" to capture anti-concentration properties. This essentially extends the idea of U-degree proposed in \cite{LT_duke}. Specifically, we use RUD to describe the “Regular Littlewood–Offord problem.” It is worth noting that a similar approach is to extend LCD to RLCD, as in \cite{Livshyts_aop}. However, the purposes of these two constructions are completely different: the main goal of RUD is to allow for a broader class of distributions. In fact, in Section \ref{Unstructured vectors section}, we can extend the properties of RUD to random variables with finite fourth moments. By contrast, RLCD is used mainly to study ``Inhomogeneous Littlewood–Offord problems" to estimate the properties of inhomogeneous random matrices.

Moreover, we prove that RUD satisfies several important properties of LCD and UD, such as lower bounds and stability. The main step in our approach is to establish anti-concentration inequalities for RUD on discrete grids, which will be presented in Section \ref{Unstructured vectors section}. Based on RUD, we then complete the preparations necessary to prove \eqref{Unstructured probability}.

Our second main innovation lies in analyzing the complement of unstructured vectors. The main obstacle is the bias term $b$ in \eqref{Standardize for discrete r.v.}, which invalidates standard concentration inequalities. In particular, the spectral norm of $M$ is typically of order $\sqrt{n}$ rather than $\sqrt{pn}$, which makes classical bounds ineffective. 

To overcome this obstacle, we control the concentration probability of $Mx$ instead of considering the coupled quantity $(M-M')x$, where $M'$ is an independent copy of $M$. 
More precisely, for all $t>0$ we have
\begin{align}\label{Coupling}
    \textsf{P}\!\left( \Vert Mx\Vert_{2} \le t \right)^{2} 
    \;\le\; \textsf{P}\!\left( \Vert (M-M')x\Vert_{2} \le 2t \right).
\end{align}
Since the spectral norm of $M-M'$ is bounded by $C\sqrt{pn}$ with high probability, 
this coupling argument allows us to establish the desired estimate in \eqref{Spread probability}.

Another critical issue is that this substitution does not hold for steep vectors in some cases. If we retain the same configuration of the Steep vectors in \cite{LT_duke}, the presence of the bias term $b$ implies that excluding the event of a zero column alone does not suffice to guarantee a lower bound for $\Vert M_{n}x\Vert_{2}$. In contrast, by adopting the preceding idea and estimating $\Vert M_{n}x\Vert_{2}$ through the inner product of $x$ with the difference of two rows, it becomes sufficient to rule out the event of a constant column. The probability of encountering a constant column is $(1+o_{n}(1))n(1-p)^{n}$. This approach is optimal only in the case $b=0$; when $b \ne 0$, however, it ceases to be optimal.

We observe that making adjustments to only a small subset of steep vectors does not affect the overall properties. Therefore, we modified the steep vectors to focus our discussion on the properties of individual columns or pairs of columns in the matrix.
First, for vectors $x$ with $x_{1}^{*} > Cn x_{2}^{*}$, we consider the inner product between the $x$ and some one column of $M_{n}$. We only need to exclude the case where $M_{n}$ contains a column consisting entirely of zeros. Furthermore, for vectors $y$ with $y_{2}^{*} > Cny_{3}^{*}$, we consider the event to be approximately two columns in $M_{n}$ linearly dependent. A natural idea would be to exclude the event where two columns are linearly dependent. However, it is in fact difficult to derive an explicit lower bound in this way. This difficulty arises because, under the preceding approach, the lower bound is determined by 
\begin{align}
    |\alpha y_{1}^{*}+\beta y_{2}^{*}|-Cny_{3}^{*}\nonumber
\end{align}
whereas linear independence only ensures that the first part is nonzero, which does not necessarily imply that the lower bound is strictly positive. Consequently, we require an alternative event that not only guarantees the lower bound but also differs only slightly in probability from the event of two columns being linearly dependent. In particular, we consider the event $\mathcal{E}_{1}$ that there exist two columns of the matrix that are approximately equal. We can obtain the lower bound by excluding this event and estimating 
\begin{align}
    \Vert My \Vert_{2} \ge \max_{i,j,k}{ \left\{ \left| \left \langle \mathrm{R}_{i}(M_{n}),y \right  \rangle  \right|,\left|  \left \langle  \mathrm{R}_{j}(M_{n})-\mathrm{R}_{k}(M_{n}),y \right \rangle \right|  \right\} }.\nonumber
\end{align}
For other parts in Steep vectors $\mathcal{T}$, we continue to use the approach from \eqref{Coupling} to estimate the lower bound of $\Vert M_{n}x\Vert_{2}$ and complete the proof of \eqref{Steep probability}.

Ultimately, we will complete the proof using a ``invertibility via distance''. This approach reduces the least singular value of the random matrix to the distance between a random vector and a random linear subspace, thus establishing its connection to the anti-concentration inequality. Related techniques have also been applied to solve other problems in nonasymptotic random matrix theory, such as \cite{Nguyen_ptrf, Nguyen_jfa, Rudelson_cpam, Rudelson_gafa, Vershynin_rsa}.

\section{Preliminaries}\label{Preliminaries}
\subsection{Decomposition of \texorpdfstring{$\mathbb{R}^{n}$}{Rn}}\label{Decomposition of Rn}
In this paper, our decomposition of $\mathbb{R}^{n}$ is analogous to the partition induced by \cite{LT_duke}, dividing the space into gradual nonconstant vectors (unstructured vectors) and the complement of that set. For any $r \in (0,1)$, we define 
\begin{align}
\Upsilon_{n}\left(r \right):=\left\{ x\in \mathbb{R}^{n} : x_{\left \lfloor rn \right \rfloor }^{*}=1   \right\}.
\end{align}
By a growth function $\textbf{g}$ we mean any nondecreasing function from $[1,\infty)$ to $[ 1, \infty)$. Let $\textbf{g}$ be a growth function, we say that a vector $x \in \Upsilon_{n}(r)$ is gradual if $x_{i}^{*} \le \textbf{g}(n/i)$ for all $i \le n$. Furthermore,  for $\delta \in (0,1)$ and $\rho >0$, we say that $x\in \Upsilon_{n}(r)$ is nonconstant if 
\begin{align}\label{nonconstant}
    \exists Q_{1} ,Q_{2} \subset [n]  \text{      such that   } \left| Q_{1} \right|,|Q_{2}| \ge \delta n \text{ and } \max_{i \in Q_{2}}x_{i} \le \min_{i \in Q_{1}}x_{i}-\rho. 
\end{align}
We now define the set $\mathcal{V}_{n}(r,\textbf{g},\delta,\rho)$ as 
\begin{align}
    \mathcal{V}_{n}:=\left\{ x \in \Upsilon_{n}(r): x \text{ is gradual with \textbf{g} and satisfies \eqref{nonconstant}}     \right\}
\end{align}
The more properties of gradual nonconstant vectors will be introduced in Section \ref{Unstructured vectors section}. We now introduce the complement of unstructured vectors. 

We first introduce our parameters. Let $d=pn$, $\gamma > 1$, $C_{1},  C_{2} > 0$, we fix a sufficiently small absolute positive constant $r$ and a sufficiently large absolute positive constant $C_{\tau}$. We also fix two integers $l_{0}$ and $s_{0} \in \mathbb{N}^{+}$ such that 
\begin{align}\label{l and s}
    l_{0}= \left\lfloor \frac{d}{4\log{1/p}} \right\rfloor \ \text{  and  } \ l_{0}^{s_{0}-1} \le (64p)^{-1} = \frac{n}{64d} < l_{0}^{s_{0}}.
\end{align}
For $1 \le j \le s_{0}$, we set 
\begin{align}
    n_{0}=2, \ \ \ n_{j}=3l_{0}^{j-1},\ \ \ n_{s_{0}+2}=\left\lfloor \sqrt{n/p} \right\rfloor \ \text{ and } \ n_{s_{0}+3}=\left\lfloor rn  \right\rfloor.
\end{align}
Then, set $n_{s_{0}+1}=\left\lfloor 1/(64p)  \right\rfloor$ if $\left\lfloor 1/(64p) \right\rfloor \ge 1.5n_{s_{0}}$. Otherwise, let $n_{s_{0}+1}=n_{s_{0}}$. 
Finally, we set $\kappa=\kappa(p) = \frac{\log{\gamma d}}{\log{l_{0}}}$.\par 
We first introduce steep vectors. Set 
\begin{align}
    \mathcal{T}_{0}:=\left\{ x \in \mathbb{R}^{n}: x_{1}^{*} \ge C_{1} nx_{2}^{*} \right\} \ \text{ and } \ \mathcal{T}_{11}:=\left\{ x \in \mathbb{R}^{n}: x\notin \mathcal{T}_{0} \text{ and } x_{2}^{*} \ge C_{2} n x_{n_{1}}^{*}  \right\}.\nonumber
\end{align}
Then, for $2 \le j \le s_{0}+1$, we set 
\begin{align}
    \mathcal{T}_{1j}=\left\{ x \in \mathbb{R}^{n}: x \notin \bigcup_{i=1}^{j-1}\mathcal{T}_{1i} \cup \mathcal{T}_{0} \text{ and } x_{n_{j-1}}^{*} \ge \gamma dx_{n_{j}}^{*}  \right\}
     \text{ and } \mathcal{T}_{1} = \bigcup_{j=1}^{s_{0}+1}\mathcal{T}_{1j}.\nonumber
\end{align}
We now set $j = j(k)= s_{0}+k$ for $k=2,3$ and define
\begin{align}
    \mathcal{T}_{k}=\left\{ x \in \mathbb{R}^{n}: x \notin \bigcup_{i=0}^{k-1}\mathcal{T}_{i} \text{ and } x_{n_{j(k)-1}}^{*} \ge C_{\tau} \sqrt{d} x_{n_{j(k)}}^{*}  \right\}.
\end{align}
The steep vectors is $\mathcal{T}= \bigcup_{i=0}^{4}\mathcal{T}_{i}$.\par 
We now introduce the ``Spread" vectors($\mathcal{R}$-vectors). Given $n_{s_{0}+1} < k \le n/\log^{2}{d}$, define $A=A(k)=[k:n]$ and the set
\begin{equation}
\begin{aligned}
    \mathcal{AC}\left( \rho \right):= \{& x\in \mathbb{R}^{n}: \exists \lambda \in \mathbb{R} \text{ such that } |\lambda|=x_{\left\lfloor rn \right\rfloor}^{*} \text{ and }\\
    & |\left\{ i \le n: |x_{i}-\lambda| \le \rho|\lambda|  \right\}| \ge n- \left\lfloor rn \right\rfloor  \}.\nonumber
\end{aligned}
\end{equation}
We now give the $\mathcal{R}$-vectors as follows:
\begin{equation}
\begin{aligned}
     \mathcal{R}_{k}^{1}: =& \bigg\{ x \in \left( \Upsilon_{n}(r) \setminus \mathcal{T} \right) \cap \mathcal{AC}(\rho): \Vert x_{\sigma_{x}(A)}\Vert_{2}/\Vert x_{\sigma_{x}(A)} \Vert_{\infty}  \ge C_{0}/\sqrt{p} \text{ and } \\
    & \sqrt{n/2} \le \Vert x_{\sigma_{x}(A)}\Vert_{2} \le C_{\tau }\sqrt{dn} \bigg\}  \nonumber
\end{aligned}
\end{equation}
and 
\begin{equation}
\begin{aligned}
     \mathcal{R}_{k}^{2}: = &  \bigg\{ x \in  \Upsilon_{n}(r) \setminus \mathcal{T} : \Vert x_{\sigma_{x}(A)}\Vert_{2}/\Vert x_{\sigma_{x}(A)} \Vert_{\infty}  \ge C_{0}/\sqrt{p} \text{ and } \\
    & \frac{2\sqrt{n}}{r} \le \Vert x_{\sigma_{x}(A)}\Vert_{2} \le C_{\tau }d\sqrt{n} \bigg\},  \nonumber
\end{aligned}
\end{equation}
where $C_{0}>0$ is fixed in Section \ref{Structured vectors section}. Define $\mathcal{R}=\bigcup_{n_{s_{0}+1}<k \le n/\log^{2}{d} }{\left( \mathcal{R}_{k}^{1} \cup \mathcal{R}_{k}^{2} \right)} $.\par 
Finally, we show that the complement of unstructured vectors belongs to $\mathcal{R} \cup \mathcal{T}$. This is the version of Theorem 6.3 in \cite{LT_duke}. We first give the growth function:
\begin{align}\label{g in structured vectors}
    \textbf{g}(t)=(2t)^{3/2} \text{ for } 1 \le t <64d \text{ and } \textbf{g}(t)=\exp{\left( \log^{2}{(2t)} \right)} \text{ for } t \ge 64d.
\end{align}
\begin{mypropo}\label{Structured vectors partition}
    For any positive constants $C_{0}$, $C_{1}$, $C_{2}$ and $\gamma$, there exist positive constants $C_{\tau}$, $C_{\ref{Structured vectors partition}}$ and $c_{\ref{Structured vectors partition}}$ depending on $C_{0}$, $C_{1}$, $C_{2}$ and $\gamma$ such that the following holds. Let $n \ge C_{\ref{Structured vectors partition}}$, $p \in (0,c_{\ref{Structured vectors partition}})$ and assume that $d = pn \ge C_{\ref{Structured vectors partition}} \log{n}$. Let $r \in (0,1/2)$, $\delta \in (0,r/3)$, $\rho \in (0,1)$ and let $\textbf{g}$ satisfies \eqref{g in structured vectors}. Then,
    \begin{align}
        \Upsilon_{n}(r) \setminus \mathcal{V}_{n}(r,\textbf{g},\delta,\rho) \subset \mathcal{R}\cup \mathcal{T}.\nonumber
    \end{align}
\end{mypropo}
\begin{proof}
    It is sufficient to prove $\Upsilon_{n}(r)\setminus (\mathcal{V}_{n}(r,\textbf{g},\delta,\rho) \cup \mathcal{R} \cup \mathcal{T} ) = \emptyset$. Following the proof of Theorem 6.3 in \cite{LT_duke} and noting that our main differences lie in \(\mathcal{T}_0\) and \(\mathcal{T}_{11}\), we conclude that for every \(x \in \Upsilon_{n}(r)\setminus (\mathcal{V}_{n}(r,\textbf{g},\delta,\rho) \cup \mathcal{R} \cup \mathcal{T} )\) there exists a smallest index \(j \in \{ 1,2 \}\) such that \(x_j^{*} \ge \exp{( \log^{2}{(2n/j)} ) } \).

    If $j=2$, we have 
    \begin{align}
        C_{\tau}^{2}d(\gamma d)^{s_{0}}C_{2}n \ge x_{2}^{*} \ge \exp{(\log^{2}{n})}.\nonumber
    \end{align}
    As a simple estimation from the proof of Theorem 6.3 in \cite{LT_duke}, we have 
    \begin{align}
        4\log{d}+ \frac{\log{\gamma d}}{\log{l_{0}}}\log{1/(64p)} \ge \log{ \left( (C_{\tau}^{2}d)(\gamma d)^{s_{0}+1} \right) }.\nonumber
    \end{align}
    We imply above inequality to obtain
    \begin{align}
        \log{\frac{C_{1}n}{\gamma d}}+ 4\log{d}+ \frac{\log{\gamma d}}{\log{l_{0}}}\log{1/(64p)} \ge \log^{2}{n},\nonumber
    \end{align}
    which is impossible.

    If $j=1$, the same argument as above shows that this cannot occur; hence the result is proved.    
\end{proof}
\subsection{Concentration properties for sparse random variables and matrix}\label{Concentration for sparse r.v.}
In this subsection, we will introduce some properties of the sparse random variables and the matrix. We begin with the simple concentration inequality.
\begin{mylem}\label{pn/8-8pn}
    Let $n \ge 1$, $50/n <p <0.1$, $\tau >e$ and $\{ \delta_{i}\}_{ i \le n}$ be Bernoulli($p$), we have
    \begin{align}
        \textsf{P}\left( \sum_{i=1}^{n}{\delta_{i}} >(\tau +1)pn \right) \le \exp{\left( -\tau\log{\left( \tau/e \right)}pn \right)}.\nonumber
    \end{align}
    Furthermore, we have
    \begin{align}
        \textsf{P}\left( pn/8 \le \sum_{i=1}^{n}{\delta_{i}} \le 8pn \right) \ge 1-\left( 1-p \right)^{n/2}.\nonumber
    \end{align}
\end{mylem}
We now give the following lemma to estimate the tail probability of the independent sum.
\begin{mylem}\label{Sum concentration}
    For $n \ge 30$, $6\log{n}/n \le p \le 1/2$. Let $M$ be an $n \times n$ random matrix satisfying the assumption of Theorem \ref{Theorem A}. There exists $C_{\ref{Sum concentration}}=C_{\ref{Sum concentration}}(S_{\xi})>0$ such that the following holds. Let 
    \begin{align}
        \mathcal{E}_{\text{sum}}:=\left\{ M= (\delta_{ij}\xi_{ij}+b): \forall i,j \in [n], \left| \mathrm{C}_{i}(M)-\mathrm{C}_{j}(M) \right| \le C_{\ref{Sum concentration}}pn \right\}.\nonumber
    \end{align}
    We have $\textsf{P}\left( \mathcal{E}_{\text{sum}} \right) \ge 1- \exp{(-2.7pn)}$.
\end{mylem}
\begin{proof}
    For fixed $k_{1}$ and $k_{2}$, we denote 
    \begin{align}
        \Omega_{k_{1},k_{2}}:= \left\{ \sum_{j=1}^{n}(\delta_{k_{1}j}+ \delta_{k_{2}j}) > 10pn \right\}.\nonumber
    \end{align}
    Since Lemma \ref{pn/8-8pn}, we have
    \begin{align}
        \textsf{P}\left( \Omega_{k_{1},k_{2}} \right) \le \exp{(-8\log{(4/e)}pn)} <\exp{(-3pn)}.\nonumber
    \end{align}
    For $C=10\max_{i \le L}|a_{i}|$, we have
    \begin{equation}
        \begin{aligned}
            \textsf{P}\left( \left| \mathrm{C}_{k_{1}}(M)-\mathrm{C}_{k_{2}}(M) \right| \ge  Cpn  \right)
            & \le  \textsf{P}\left( \sum_{j=1}^{n}{(\delta_{k_{1}j}|\xi_{k_{1}j}|+\delta_{k_{2}j}|\xi_{k_{2}j}|)} \ge Cpn  \right)\\
            & \le \textsf{P}\left( \max_{i\le L}|a_{i}|\sum_{j=1}^{n}(\delta_{k_{1}j}+\delta_{k_{2}j}) \ge Cpn  \right)\\
            &  \le e^{-3pn}.  \nonumber
        \end{aligned}
    \end{equation}
    Finally, we obtain
    \begin{align}
        \textsf{P}\left( \mathcal{E}_{\mathrm{sum}}^{c}  \right) \le n^{2}\exp{(-3pn)} \le \exp{(-2.7pn)},\nonumber
    \end{align}
    which implies the result.
\end{proof}
We also need the following simple lemma from Litvak and Tikhominov in \cite{LT_duke}
\begin{mylem}\label{all columns concentration}
For any $ R \ge 1 $, there is $ C_{\ref{all columns concentration}} = C_{\ref{all columns concentration}}(R) \ge 1 $ such that the following holds. Let $ n \ge 1$ and $ p \in (0,1) $ satisfy $ C_{\ref{all columns concentration}}p \le 1 $ and $ C_{\ref{all columns concentration}} \le pn $. Furthermore, let $ M $ be an $ n \times n $  random matrix with i.i.d. Bernoulli(p) entries. Then with probability at least $ 1 - \exp(-n/C_{\ref{all columns concentration}}) $ one has
\begin{align}
8pn \ge |\mathrm{supp} C_i(M)| \ge pn/8 \text{ for all but } \lfloor(pR)^{-1}\rfloor \text{ indices } i \in [n] \setminus \{1\}.\nonumber
\end{align}
\end{mylem}

Next, we introduce the tail probability of the spectral norm of the random matrix.
\begin{mylem}\label{Spectral norm}
    Let $n$ be a large enough integer, $6\log{n}/n \le p \le 1/2$, $b \in \mathbb{R}$, $\delta$ be Bernoulli(p), and $\xi$ satisfy \eqref{Condition}. Let $M$ be a $n \times n$ random matrix with i.i.d. entries with $\delta \xi$. For any $R \ge 1$, there exists $C_{\ref{Spectral norm}}=C_{\ref{Spectral norm}}(R,\xi)>1$ such that 
    \begin{align}
        \textsf{P}\left( \Vert M-\textsf{E}M\Vert \ge C_{\ref{Spectral norm}}\sqrt{pn} \right) \le \exp{(-Rpn)}.\nonumber
    \end{align}
\end{mylem}
\begin{proof}
      We first define the random matrix
      \begin{align}
          A=(\delta_{ij}(\xi_{ij}-\textsf{E}\xi_{ij})) \text{ and } H=(\textsf{E}[\xi_{ij}]\delta_{ij}).\nonumber
      \end{align}
      By the proof of Theorem 1.7 in \cite{BR_adv} and Lemma 3.6 in \cite{LT_duke}.
      We have for any $R\ge 1 $, there exists $C_{\ref{Spectral norm}}=C_{\ref{Spectral norm}}(\xi,R)>1$ such that
      \begin{align}
          \textsf{P}\left( \Vert M-\textsf{E}M\Vert \ge C_{\ref{Spectral norm}}\sqrt{pn} \right) \le e^{-Rpn}.\nonumber
      \end{align}
\end{proof}

\subsection{Anti-concentration}\label{Anti concentration}
In this subsection, we give two anticoncentration inequalities. Firstly, we introduce the following tensorization lemma.
\begin{mylem}\label{tensorization}
    Let $\lambda, \gamma > 0$ and $ (\xi_1, \xi_2, \ldots, \xi_m)$ be independent random variables. Assume that for all $j \leq m$, we have 
\begin{align}
    \textsf{P}(|\xi_j| \leq \lambda) \leq \gamma.\nonumber
\end{align} 
Then for every $\varepsilon \in (0, 1)$ one has  
\begin{align}
    \textsf{P}\left( \|(\xi_1, \xi_2, \ldots, \xi_m)\| \leq \lambda \sqrt{\varepsilon m} \right) \leq (\varepsilon / \varepsilon)^{\varepsilon m} \gamma^{m(1-\varepsilon)}.\nonumber
\end{align}
Moreover, if there exists $\varepsilon_0 > 0$ and $K > 0$ such that for every $\varepsilon \geq \varepsilon_0$ and for all $j \leq m$ one has  
\begin{align}
    \textsf{P}(|\xi_j| \leq \varepsilon) \leq K\varepsilon,\nonumber
\end{align}  
then there exists an absolute constant $C_{\ref{tensorization}} > 0$ such that for every $\varepsilon \geq \varepsilon_0$,  
\begin{align}
    \mathbb{P}\left( \|(\xi_1, \xi_2, \ldots, \xi_m)\| \leq \varepsilon \sqrt{m} \right) \leq (C_{\ref{tensorization}} K \varepsilon)^m.\nonumber
\end{align}
\end{mylem}
Next, we need the following inequality for the L\'{e}vy concentration function of the independent sum.
\begin{mylem}\label{Levy concentration}
    There exist $C_{0}>0$ and $C_{\ref{Levy concentration}}>0$ depending on $\xi$ satisfying \eqref{Condition}. For $A \subset [n]$ and $x \in \mathbb{R}^{n}$ be such that $\Vert x_{A}\Vert_{\infty} \le C_{0}^{-1}\sqrt{p}\Vert x_{A}\Vert_{2}$. Then for $\{ \delta_{j}  \}_{j \le n}$ be Bernoulli(p) and $\{ \xi_{j} \}$ be i.i.d random variables satisfies \eqref{Condition}, we have
    \begin{align}
        \mathcal{L}\left( \sum_{i=1}^{n}{\delta_{i}\xi_{i}} , C_{\ref{Levy concentration}}\sqrt{p}\Vert x_{A}\Vert_{2} \right) \le e^{-8}.\nonumber
    \end{align}
    Moreover, for $b \in \mathbb{R}$ and $m \in \mathbb{N}$. Let $M$ be an $m \times n$ random matrix with i.i.d. entries with $\delta \xi$ and $M'$ be an independent copy of $M$, then there exists $C'_{\ref{Levy concentration}}>0$ depending on $T$ and $B$ such that 
    \begin{align}
        \textsf{P}\left( \Vert (M-M')x \Vert_{2} \le C'_{\ref{Levy concentration}}\sqrt{pm}\Vert x_{A}\Vert_{2} \right) \le \exp{(-6m)}.\nonumber
    \end{align}
\end{mylem}
\begin{myrem}
    The proof of this lemma can be carried out similarly to that of Lemma 3.5 in \cite{BR_adv} and is omitted here.
\end{myrem}
\section{Unstructured vectors}\label{Unstructured vectors section}
The main goal of this section is to prove that the probability that the small ball probability of the unstructured vectors is small is sup-exponentially small. Therefore, on the one hand we need to provide small ball probability estimates for the type of random variables considered in this paper, and on the other hand we must show that the probability of vectors with small small ball probability existing in the kernel of the matrix is exponentially small.  We will follow an approach similar to the ``U-degree" in \cite{LT_duke}.  In the first subsection we will establish the small ball probability estimate; afterwards, we will present auxiliary lemmas in separate subsections and give the final proof in the last section.
\subsection{Small ball probability via Randomized U-degree}\label{Small ball section}
The main purpose of this subsection is to present a method to estimate the probability of small ball using the ``Randomized Unstructured Degree"(RUD). Before formally introducing the RUD, we need to establish several preliminary concepts.\par 
First, for any finite subset $S \subset \mathbb{Z}$, we define $\eta[S]$ as a random variable uniformly distributed over $S$. Next, for any $K_{2} \ge 1$, we introduce a smooth cutoff function $\psi_{K_{2}} (t)$ satisfying the following conditions:
\begin{itemize}
    \item the function $\psi_{K_{2}}$ is twice continuously differentiable, with $\Vert \psi_{K_{2}}\Vert_{\infty} = 1$ and $\Vert \psi_{K_{2}}'' \Vert_{\infty} < \infty$;
    \item $\psi_{K_{2}}(t) = \frac{1}{K_{2}}$ for all $ t \le \frac{1}{2K_{2}}$;
    \item $ \frac{1}{K_{2}} \ge \psi_{K_{2}}(t) \ge t$ for all $ \frac{1}{K_{2}} \ge t \ge \frac{1}{2K_{2}}$;
    \item $\psi_{K_{2}}(t ) =t$ for all $t \ge \frac{1}{K_{2}}$.
\end{itemize}
Finally, we provide the assumption of the growth function $\textbf{g}(t)$: $[1,\infty ) \to [0,\infty )$ as follows
\begin{align}\label{Condition g}
    \forall a \ge 2,t \ge 1: \textbf{g}(at) \ge \textbf{g}(t)+a \ \text{ and } \ \prod_{j=1}^{\infty}\textbf{g}(2^{j})^{j2^{-j}} \le K_{3},
\end{align}
where $K_{3} \ge 1$.\par
Now we give the definition of the ``Randomized Unstructured Degree".
\begin{mydef}\label{Def: U-degree}
    Given $n \in \mathbb{N}^{+}$, $1 \le m \le n/2$, $b \in \mathbb{R}$, $y \in \mathbb{R}^{n}$, random variable $\xi$ satisfying \eqref{Condition} and parameters $K_{1},K_{2} \ge 1$, we define the ``Randomized Unstructured Degree" of $y$ and $\xi$ by 
    \begin{equation}
    \begin{aligned}
         &\mathrm{UD}_{n}^{\xi} (m,y,K_{1},K_{2})\\
         &    := \sup{ \left\{ t >0: A_{nm}\sum_{S_{1},\dots,S_{m}}{ 
         \int_{-t}^{t}{  \prod_{i=1}^{m}\psi_{K_{2}}\left( \left| \textsf{E}\exp{\left( 2\pi \textbf{i} y_{\eta[S_{i}]}\xi m^{-1/2} s  \right)}  \right|  \right) } \mathrm{d}s   }   \le K_{1} \right\}    },\nonumber
    \end{aligned}
    \end{equation}
    where the sum is taken over all sequences of disjoint sets $S_{1},\dots,S_{m} \subset [n]$ with cardinality $\left\lfloor n/m \right\rfloor$ and $A_{nm}^{-1}$ is the cardinality of all sequences, which implies:
    \begin{align*}
        A_{nm}:= \frac{\left( \left( \left\lfloor n/m \right\rfloor \right)!  \right)^{m} \left( n-\left\lfloor n/m \right\rfloor \right)! }{n!}.
    \end{align*}
\end{mydef}
Using the definition of RUD together with the Ess\'{e}en lemma, we can now present the proof of the following theorem, which constitutes the main result of this subsection.
\begin{mytheo}\label{Small ball prob via UD}
    Let $n ,m \in \mathbb{N}^{+}$ with $m \le n/2$ and $K_{1},K_{2} \ge 1$. Let $v \in \mathbb{R}^{n}$, $b \in \mathbb{R}$ and $\xi $ be a random variable that satisfies \eqref{Condition}. For $X= (X_{1}, \dots, X_{n})$ is a random vector uniformly distributed on the set of vectors with $m$ ones and $n-m$ zeros and $Y:=(Y_{1},\dots, Y_{n})=(X_{1}\xi_{1}+b,\dots,X_{n}\xi_{n}+b)$, where $\xi_{i}$ are i.i.d. random variables with $\xi$. Then for all $\tau \ge 0$, we have,
    \begin{align}
        \mathcal{L}\left( \sum_{i=1}^{n}{v_{i}Y_{i}},\sqrt{m}\tau \right) \le C_{\ref{Small ball prob via UD}}\left( \tau+ \mathrm{UD}_{n}^{\xi}(m,v,K_{1},K_{2})^{-1} \right),\nonumber
    \end{align}
    where $C_{\ref{Small ball prob via UD}}>0$ depending on $K_{1}$.
\end{mytheo}
\begin{myrem}\label{Finite fourth moment}
    The above conclusion remains valid for random variables with finite fourth moments. Moreover, for any such random variable $\xi$, its RUD continues to satisfy Proposition \ref{Anti-concentration of U-degree in Lambda} stated in this section. To justify this claim, note that the proofs of Proposition \ref{Anti-concentration of U-degree in Lambda} given here rely mainly on the finiteness of the second moments and standard properties of expectation; since the arguments are nearly identical to those used in Section 4 in \cite{LT_duke}, we omit the details in this section. The only essential differences arise in Lemmas \ref{Moderately small product} and \ref{Integration in small interval}; we will explicitly indicate how the proofs of these two lemmas can be carried out under the sole assumption that $\xi$ has finite fourth moment in Lemma \ref{Distance from integral}.
\end{myrem}
\begin{proof}
    For any disjoint subsets $S_{1},\dots,S_{m} \subset [n]$ with cardinality $\left\lfloor n/m \right\rfloor$, set 
    \begin{align}
        \Omega_{S_{1},\dots,S_{m}}:=\left\{ \mathrm{supp}X \cap S_{i}=1  \text{ for all } i \le m \right\}.\nonumber
    \end{align}
    We have for any $\tau>0$ and $w \in \mathbb{R}$,
    \begin{align}
        \textsf{P}\left( \left|\sum_{i=1}^{n}v_{i}Y_{i} -w \right| \le \tau \right) =A_{nm}\sum_{S_{1},\dots,S_{m}}{\textsf{P}\left( \left| \sum_{i=1}^{n}v_{i}Y_{i} -w \right| \le \tau \bigg| \Omega_{S_{1},\dots,S_{m}}  \right) }.\nonumber
    \end{align}
    Furthermore, conditioned on $\Omega_{S_{1},\dots,S_{m}}$, we have 
    \begin{align}
        \sum_{i=1}^{n}{v_{i}Y_{i}}=\sum_{i=1}^{m}v_{\eta[S_{i}]}\xi_{\eta[S_{i}]}+b\sum_{i=1}^{n}v_{i}.\nonumber
    \end{align}
    By the Ess\'{e}en lemma and the independence of $\xi_{i}$ and $\eta[S_{i}]$, we immediately have 
    \begin{equation}
    \begin{aligned}
        \mathcal{L}\left( \sum_{i=1}^{m}v_{\eta[S_{i}]}\xi_{\eta[S_{i}]},\tau \right) & \le C\int_{-1}^{1}{\prod_{i=1}^{m}{ \left| \textsf{E}\exp{ \left( 2\pi \textbf{i} v_{\eta[S_{i}]}\xi_{\eta[S_{i}]}s/\tau \right) }  \right|  }    } \mathrm{d}s \\
        & = Cm^{-1/2}\tau\int_{-\sqrt{m}/\tau}^{\sqrt{m}/\tau}{  \prod_{i=1}^{m}{ \left| \textsf{E}\exp{ \left( 2\pi \textbf{i} v_{\eta[S_{i}]}\xi m^{-1/2}s \right) }  \right|  }  } \mathrm{d}s.\nonumber
    \end{aligned}
    \end{equation}
    Let $\tau = \sqrt{m}/\mathrm{UD}_{n}^{\xi}$, we obtain:
    \begin{equation}
    \begin{aligned}
        \mathcal{L}\left( \sum_{i=1}^{n}{v_{i}Y_{i}},\tau \right)
        & \le A_{nm}\sum_{S_{1},\dots,S_{m}}{\mathcal{L}\left( \sum_{i=1}^{m}v_{\eta[S_{i}]}\xi_{\eta[S_{i}]},\tau \bigg| \Omega_{S_{1},\dots,S_{m}} \right) } \\
        & \le \frac{CA_{nm}}{\mathrm{UD}_{n}}\sum_{S_{1},\dots,S_{m}}\int_{-\mathrm{UD}_{n}}^{\mathrm{UD}_{n}}{  \prod_{i=1}^{m}{ \left| \textsf{E}\exp{ \left( 2\pi \textbf{i} v_{\eta[S_{i}]}\xi m^{-1/2}s \right) }  \right|  }  } \mathrm{d}s\\
        & \le CK_{1}/\mathrm{UD}_{n}.\nonumber
    \end{aligned}
    \end{equation}
    We now complete the proof of this theorem.
\end{proof}
\subsection{Auxiliary results}\label{Auxiliary results section}
The main purpose of this subsection is to provide several auxiliary lemmas concerning gradual nonconstant vectors and related matters, thereby facilitating the proofs of the main results in the following two subsections. Part I focuses on observations about the properties of gradual non-constant vectors, while Part II presents several additional lemmas. We begin with the following definition.

For a permutation $\sigma \subset \prod_{n}$, two disjoint subsets $Q_{1}$, $Q_{2}$ of cardinality $\left\lceil \delta n  \right\rceil$ and a number $h \in \mathbb{R}$ such that
\begin{align}\label{Lambda h}
    \forall i \in Q_{1}:h+2 \le \textbf{g}(n/\sigma^{-1}(i)) \text{ and }  \forall i \in Q_{2}: -\textbf{g}(n/\sigma^{-1}(i)) \le h-\rho -2.
\end{align}

Define the sets $\Lambda_{n}:=\Lambda_{n}(k,\textbf{g},Q_{1},Q_{2},\rho,\sigma,h)$ by 
\begin{align}
    \Lambda_{n}=\left\{ x \in \frac{1}{k}\mathbb{Z}^{n}:|x_{\sigma(i)}| \le \textbf{g}(n/i) \  \forall i \le n, \ \min_{i \in Q_{1}}x_{i} \ge h \text{ and } \max_{i \in Q_{2}}x_{i}  \le h-\rho  \right\}.\nonumber
\end{align}
In what follows, we adopt the convention that $\Lambda_{n}=\emptyset$ if $h$ does not satisfy \eqref{Lambda h}.

We present in the following lemma on the approximation of $\mathcal{V}_{n}\left( r,\textbf{g},\delta,\rho \right)$ by $\Lambda_{n}$, which is the version of Lemma 4.7 and 4.8 in \cite{LT_duke}.
\begin{mylem}\label{Lambda approximation}
    There exists an absolute constant $C_{\ref{Lambda approximation}} \ge 1$ such that the following holds. There exists a subset $ \overline{\prod}_{n} \subset \prod_{n} $ with cardinality at most $\exp{(C_{\ref{Lambda approximation}}n)}$ with the following property. For any $x \in \mathcal{V}_{n}(r,\textbf{g},\delta,\rho)$, $k \ge 4/\rho$ and any $y \in \frac{1}{k}\mathbb{Z}^{n}$ with $\Vert x-y \Vert_{\infty} \le 1/k$, we have 
    \begin{align}
        y \in \bigcup_{t = \left\lfloor -4\textbf{g(6n)}/\rho \right\rfloor}^{\left\lfloor 4\textbf{g(6n)}/\rho \right\rfloor}\bigcup_{\sigma
         \in \overline{\prod}_{n}} \bigcup_{|Q_{1}|,|Q_{2}| =\left\lceil \delta n \right\rceil } \Lambda_{n}\left(k,\textbf{g}(6\cdot),Q_{1},Q_{2},\rho/4,\sigma,\rho t/4  \right).\nonumber
    \end{align}
\end{mylem}
We also have the following lemma to estimate the size of $\Lambda_{n}$, which is introduced by \cite{LT_duke}.
\begin{mylem}\label{Size of Lambda approximation}
    Let $k \ge 1$, $h \in \mathbb{R}$, $\rho , \delta \in (0,1)$, $Q_{1}, Q_{2} \subset [n]$ with $|Q_{1}|, |Q_{2}| = \left\lceil \delta n \right\rceil$ and $\textbf{g}$ satisfies \eqref{Condition g} with $K_{3} \ge 1$. Then there exists an constant $C_{\ref{Size of Lambda approximation}}$ depending only on $K_{3}$ such that $\left| \Lambda_{n}(k,\textbf{g},Q_{1},Q_{2},\rho,\sigma,h)  \right| \le \left( C_{\ref{Size of Lambda approximation}}k \right)^{n}$.
\end{mylem}

Next, we need to introduce two integral forms of the Markov inequality, which will play a key role in the proof of the next subsection.
\begin{mylem}\label{Markov integral 1}
    For any $s \in [a,b]$, let $\xi(s)$ be a nonnegative random variable with $\xi(s) \le 1$ a.e. Assume that $\xi(s)$ is integrable on $[a,b]$ with probability $1$. If there exists an integrable random function $\phi(s)$: $[a,b] \to [0,\infty)$ satisfies for some $\varepsilon >0$ and any $s \in [a,b]$, we have
    \begin{align}
        \textsf{P}\left( \xi(s) \le \phi(s) \right) \ge 1-\varepsilon .\nonumber
    \end{align}
    Then for any $t >0$, we have 
    \begin{align}
        \textsf{P}\left( \int_{a}^{b}{ \xi(s)  } \mathrm{d}s \ge \int_{a}^{b}{ \phi(s)} \mathrm{d}s +t(b-a)   \right) \le \varepsilon/t.\nonumber
    \end{align}
\end{mylem}
\begin{mylem}\label{Markov integral 2}
    Let $I$ be a finite set, and for any $i \in I$, let $\xi_{i}$ be a nonnegative random variable with $\xi_{i} \le 1$ a.e. If there exists an integrable random function $\phi(s)$: $I \to [0,\infty)$ satisfies for some $\varepsilon >0$ and any $i \in I$, we have
    \begin{align}
        \textsf{P}\left( \xi_{i} \le \phi(i) \right) \ge 1- \varepsilon.\nonumber
    \end{align}
    Then for any $t >0$, we have 
    \begin{align}
        \textsf{P}\left( \frac{1}{|I|}\sum_{i \in I}{\xi_{i}} \ge \frac{1}{|I|}\sum_{i \in I}\phi(i)+t  \right) \le \varepsilon/t.\nonumber
    \end{align}
\end{mylem}
Our next lemma introduces Lipschitzness for the products of smoothly truncated functions $\psi_{K_{2}}(\cdot)$.
\begin{mylem}\label{Lipschitzness of the product}
    Let $y_{1},\dots,y_{m} \in \mathbb{R}$, and set $y := \max_{w \in [m]}{y_{w}}$. Let $\xi $ be a random variable that satisfies \eqref{Condition} and $S_{1}, \dots,S_{m}$ be some disjoint subsets of $[n]$. For $i \le m$ denote
    \begin{align}
        f_{i}(s):=\psi_{K_{2}}\left( \left| \frac{1}{|S_{i}|}\sum_{w \in S_{i}}{\textsf{E}\exp{(2\pi \textbf{i}y_{w}\xi s )}}  \right| \right), \ \text{ and let } \ f(s):=\prod_{i=1}^{m}{f_{i}(s)}.\nonumber 
    \end{align}
    Then $f$ is $(C_{\ref{Lipschitzness of the product}})ym$ -Lipschitz, where $C_{\ref{Lipschitzness of the product}}>0$ depends only on $\xi$ and $K_{2}$.
\end{mylem}
\begin{proof}
    Note that $\psi_{K_{2}}$ is 1-Lipschitz and for any $s$ and $t$, we have
    \begin{equation}
    \begin{aligned}
        & |\sum_{w\in S}{\textsf{E}\exp{(2\pi \textbf{i}y_{w}\xi s)}}|-|\sum_{w\in S}{\textsf{E}\exp{(2\pi \textbf{i}y_{w}\xi t)}}|\\
        & \le \sum_{w \in S}\textsf{E}|\exp{( 2 \pi \textbf{i}\xi y_{w} s )} -\exp(2\pi \textbf{i} \xi y_{w} t)| \\
        & \le 2 \pi |S|\textsf{E}|\xi| y|s-t|,\nonumber
    \end{aligned}
    \end{equation}
    Which implies that $f_{i}$ is $C(\xi)y$-Lipschitz. Since $|\sum_{w \in S_{i}}{\textsf{E}\exp{(2\pi \textbf{i} \xi y_{w}s )}} | \le |S_{i}|$, we have $1/(2K_{2}) \le f_{i} \le 1$. Thus, for any $s ,\Delta s \in \mathbb{R}$, 
    \begin{align}
        \frac{f_{i}(s)}{f_{i}(s+\Delta s)} = 1+\frac{f_{i}(s)-f_{i}(s+\Delta s)}{f_{i}(s+\Delta s)} \le 1 + C(\xi)K_{2}y|\Delta s|.\nonumber
    \end{align}
    Furthermore, we provide that for the product of the $f_{i}$
    \begin{align}
        \frac{f(s)}{f(s+\Delta s)} \le (1+ C(\xi)K_{2}y|\Delta s|)^{m} \le 1+ C_{\ref{Lipschitzness of the product}}ym|\Delta s|,\nonumber
    \end{align}
    which implies our result.
\end{proof}
Finally, we give a simple combinatorial estimate from \cite{LT_duke}.
\begin{mylem}\label{Combinatorial estimate for Q1 and Q2}
    For any $\delta \in (0,1]$ there exist $n_{\ref{Combinatorial estimate for Q1 and Q2}} \in \mathbb{N}$, $c_{\ref{Combinatorial estimate for Q1 and Q2}}>0$ and $C_{\ref{Combinatorial estimate for Q1 and Q2}} \ge 1$ depending only on $\delta$ such that the following holds. Let $n \ge n_{\ref{Combinatorial estimate for Q1 and Q2}}$ and $m \in \mathbb{N}$ with $n /m \ge C_{\ref{Combinatorial estimate for Q1 and Q2}}$. Denote by $\mathcal{J}$ the collection of disjoint sequences $(S_{1},\dots,S_{m})$ with cardinality $\left\lfloor n/m  \right\rfloor$. Then for any disjoint subset $|Q_{1}|, |Q_{2}| \subset [n]$ with cardinality at least $\delta n$ we have
    \begin{equation}
    \begin{aligned}
         \bigg| \bigg\{ \left( S_{1},\dots,S_{m} \right) &  \in \mathcal{J}:\min{\left( |Q_{1}\cap S_{i}| ,|Q_{2}\cap S_{i}|  \right) } \ge \delta\left\lfloor n/m \right\rfloor/2 \text{ for at most }\\
        & c_{\ref{Combinatorial estimate for Q1 and Q2}}m \text{ indices } i  \bigg\} \bigg| \le e^{-c_{\ref{Combinatorial estimate for Q1 and Q2}}n}A_{nm}^{-1}.\nonumber
    \end{aligned}
    \end{equation}
\end{mylem}
\subsection{Anti-concentration of the Randomized U-degree}\label{Anti-concentration of U-degree section}
The goal of this subsection is to prove the anti-concentration of RUD in $\Lambda_{n}$. We first fix $\rho, \delta \in (0,1/4)$, a growth function $\textbf{g}$ satisfying \eqref{Condition g}, a permutation $\sigma \in \prod_{n}$, a number $h \in \mathbb{R}$, two subsets $Q_{1},Q_{2} \subset [n]$ such that $|Q_{1}|=|Q_{2}|=\left\lceil  \delta n \right\rceil$ and a random variable $\xi$ satisfying \eqref{Condition}.\par
We first give our main result in this subsection.
\begin{mypropo}\label{Anti-concentration of U-degree in Lambda}
    Let $\varepsilon \in (0,1/8)$, $\rho, \delta \in (0,1/4)$, the growth function $\textbf{g}$ satisfies \eqref{Condition g}, and random variable $\xi$ satisfies \eqref{Condition}. Then there exist $K_{\ref{Anti-concentration of U-degree in Lambda}}=K_{\ref{Anti-concentration of U-degree in Lambda}}(\xi,\delta,\rho) \ge 1$, $n_{\ref{Anti-concentration of U-degree in Lambda}}=n_{\ref{Anti-concentration of U-degree in Lambda}}(\xi,\delta,\rho,\varepsilon,K_{3}) \in \mathbb{N}$ and $C_{\ref{Anti-concentration of U-degree in Lambda}}=C_{\ref{Anti-concentration of U-degree in Lambda}}(\xi,\delta,\rho,\varepsilon,K_{3}) \in \mathbb{N}$ such that the following holds. Let $\sigma \in \prod_{n}$, $h \in \mathbb{R}$ and $Q_{1},Q_{2} \subset [n]$ with cardinality $\left\lceil  \delta n \right\rceil$. Let $8 \le K_{2} \le 1/\varepsilon$, $n \ge n_{\ref{Anti-concentration of U-degree in Lambda}}$, $m \ge C_{\ref{Anti-concentration of U-degree in Lambda}}$ with $n /m \ge C_{\ref{Anti-concentration of U-degree in Lambda}}$, $1 \le k \le \min{ \left( (K_{2}/8)^{m/2},2^{n/C_{\ref{Anti-concentration of U-degree in Lambda}}}  \right)}$ and let $X=(X_{1},\dots,X_{n})$ be a random vector uniformly distributed on $\Lambda_{n}\left( k,\textbf{g},Q_{1},Q_{2},\rho,\sigma,h  \right)$. Then, we have
    \begin{align}
        \textsf{P}\left( \mathrm{UD}_{n}^{\xi}\left( m,X,K_{\ref{Anti-concentration of U-degree in Lambda}},K_{2} \right) \le km^{1/2}/C_{\ref{Anti-concentration of U-degree in Lambda}}  \right) \le \varepsilon^{n}.\nonumber
    \end{align}
\end{mypropo}
To establish this proposition, we need to analyze the integral behavior over the interval $[-k\sqrt{m}/C_{\ref{Anti-concentration of U-degree in Lambda}},k\sqrt{m}/C_{\ref{Anti-concentration of U-degree in Lambda}}]$ in the definition of $\mathrm{UD}_{n}^{\xi}$. We decompose it into two parts: a central subinterval and two edge subintervals. For the edges, we show that, on the one hand, the collection of sets for which the products of $\psi_{K_{2}}(\cdot)$ are exponentially small occupies the vast majority, while for the remaining small fraction of sets in the edge regions, the products remain bounded.

We give the first part: the product of $\psi_{K_{2}}$ is small enough except for a set with measure $O(1)$. The proof of this lemma is lengthy and almost identical to that of Lemma 4.17 in \cite{LT_duke}, so we omit the detailed proof and state the lemma as follows.
\begin{mylem}\label{Small product}
    For any $\varepsilon \in (0,1/2)$ there are $R_{\ref{Small product}}=R_{\ref{Small product}}(\varepsilon,S_{\xi}) \ge 1$, $l_{\ref{Small product}}=l_{\ref{Small product}}(\varepsilon,S_{\xi}) \in \mathbb{N}$, and $n_{\ref{Small product}}=n_{\ref{Small product}}(\varepsilon,S_{\xi},K_{3})$ such that the following holds. Let $k,m,n \in \mathbb{N}^{+}$, $n\ge n_{\ref{Small product}}$, $k \le 2^{n/l_{\ref{Small product}}}$, $n/m \ge l_{\ref{Small product}}$, and $4 \le K_{2} \le 2/\varepsilon$. Let $X=(X_{1},\dots,X_{n})$ be random vectors uniformly distributed on $\Lambda_{n}$. Fix disjoint subsets $S_{1},\dots,S_{m} \subset [n]$ with cardinality $\left\lfloor n/m \right\rfloor$. Then the probability of 
    \begin{align}
        \left\{ \left| \left\{  s\in [0,k/2]: \prod_{i=1}^{m}{ \psi_{K_{2}}(\left| \frac{1}{\left\lfloor n/m \right\rfloor}\sum_{w \in S_{i}}\textsf{E}_{\xi}\exp{(2\pi \textbf{i}\xi X_{w}s   )} \right|)} \ge \left( K_{2}/4 \right)^{-m/2}  \right\} \right| \le R_{\ref{Small product}}  \right\} \nonumber
    \end{align}
    at least $1-\left( \varepsilon/2 \right)^{n}$.
\end{mylem}
Before presenting the final parts of the proofs of two lemmas, we first state the following crucial lemma, which serves as the key to both subsequent arguments. Moreover, because the random variables involved are only assumed to have bounded fourth moments, this lemma also substantiates the conclusion stated in Remark \ref{Finite fourth moment}.
\begin{mylem}\label{Distance from integral}
Let $\varepsilon \in (0,1)$, $s \in \mathbb{R}^{+}$, $k \in \mathbb{N}$, $h_{1},h_{2},h\in \mathbb{R}$ with $h=h_{2}-h_{1}>0$ and $\zeta$ be a random variable with a finite fourth moment. Then for $X$ is a random variable uniformly distributed on $\frac{1}{k}\mathbb{Z}\cap [h_{1},h_{2}]$, we have
    \begin{align}\label{Distance inequality}
        \textsf{P}\left( \textsf{E}_{\zeta}\mathrm{dist}\left( \zeta sX,\mathbb{Z} \right) \le \varepsilon \right) \le C_{\zeta}f(\varepsilon,s,h,k), 
    \end{align}
    where $C_{\zeta}>0$ is a constant depending only on $\zeta$ and $f(\varepsilon,s,h,k)$ is denoted by 
    \begin{align}
        f(\varepsilon,s,h,k):=\max{\left\{  \frac{1}{kh},\frac{\varepsilon}{sh},\varepsilon,\frac{s}{k}   \right\} }.\nonumber
    \end{align}
\end{mylem}
\begin{myrem}
    Indeed, this theorem specifies the class of lattices on which RUD satisfies anticoncentration; similarly, the deeper reason that LCD in \cite{Rudelson_adv} , RLCD in \cite{Livshyts_aop} and UD in \cite{LT_duke} previously enjoyed anticoncentration was likewise an estimate of the distance from a random variable to the integers.
\end{myrem}
\begin{proof}
    We first replace $\zeta$ in \eqref{Distance inequality} with an arbitrary $w \in [a,b]$, where both $a$ and $b$ are positive real numbers.\par 
    Set $t=ws\in [as,bs]$, $P_{1}=\left\lfloor th_{1} \right\rfloor$ and $P_{2}=\left\lceil th_{2} \right\rceil$, we have
    \begin{equation}
    \begin{aligned}
        \textsf{P}\left( \mathrm{dist}(tX,\mathbb{Z}) \le \varepsilon \right) 
        & \le \sum_{i=P_{1}}^{P_{2}}\textsf{P}\left( |tX-i| \le \varepsilon  \right)\\
        & \le \sum_{i=P_{1}}^{P_{2}}\textsf{P}\left( |kX-ki/t |\le \frac{k\varepsilon}{t}  \right)\\
        & \le (t(h_{2}-h_{2})+2)\frac{2k\varepsilon + t}{kt(h_{2}-h_{1})}\\
        & \le C_{a,b}f(\varepsilon,s,h,k),\nonumber
    \end{aligned}
    \end{equation}
    where $C_{a,b}>0$ depending on $a$ and $b$.\par 
    Now, we return to the proof of \eqref{Distance inequality}. Applying Paley-zygmund inequality for $\zeta^{2}$, we have 
    \begin{align}
        \textsf{P}\left( |\zeta| \ge \frac{1}{2}\sqrt{\textsf{E}\zeta^{2} } \right) \ge \frac{9(\textsf{E}\zeta^{2})^{2}}{16\textsf{E}\zeta^{4}}:=c_{0}.\nonumber
    \end{align}
    Otherwise, by Markov's inequality, we obtain
    \begin{align}
        \textsf{P}\left( |\zeta| \ge u\sqrt{\textsf{E}\zeta^{2}}  \right) \le \frac{1}{u^{2}}.\nonumber
    \end{align}
    Furthermore, there exist $0< a < b$ such that 
    \begin{align}
        \textsf{P}\left(\mathcal{E} \right) :=\textsf{P}\left( |\zeta| \in [a,b] \right) \ge c_{0}/2.\nonumber
    \end{align}
    Thus, we have 
    \begin{equation}
    \begin{aligned}
        \textsf{P}\left(  \textsf{E}_{\zeta}\mathrm{dist}\left( \zeta s X,\mathbb{Z}  \right) \le \varepsilon  \right) \le \textsf{P}\left( \textsf{E}_{\mathcal{E}}\mathrm{dist}(|\zeta| s X,\mathbb{Z}) \le 2\varepsilon/c_{0} \right) \le C_{\zeta}f(\varepsilon,s,h,k).\nonumber
    \end{aligned}
    \end{equation}
\end{proof}
Next, we provide the other part in the edge subintervals.
\begin{mylem}\label{Moderately small product}
    For any $\varepsilon \in (0,1)$ and $z \in (0,1)$ there are $\varepsilon' \in \varepsilon'(\varepsilon,\xi,z) \in (0,1/2)$, $n_{\ref{Moderately small product}}=n_{\ref{Moderately small product}}(\varepsilon,z,\xi) \ge 10$, and $C_{\ref{Moderately small product}}=C_{\ref{Moderately small product}}(\varepsilon,z,\xi) \ge 1$ such that the following holds. Let $n\ge n_{\ref{Moderately small product}}$, $2^{n} \ge k \ge 1$, $C_{\ref{Moderately small product}} \le m \le n/4$, and $4 \le K_{2} \le 1/\varepsilon$. Let $X=(X_{1},\dots,X_{n})$ be a random vector uniformly distributed on $\Lambda_{n}$. Fix disjoint subsets $S_{1},\dots,S_{m}$ with cardinality $\left\lfloor n/m \right\rfloor$. Then the probability of the following event:
    \begin{align}
        \left\{ \forall s \in [z,\varepsilon'k]:\prod_{i=1}^{m}{ \psi_{K_{2}}\left(\left|  \frac{1}{\left\lfloor n/m \right\rfloor}\sum_{w \in S_{i}}{\textsf{E}_{\xi}\exp(2\pi \textbf{i}X_{w}\xi s )}  \right| \right) } \le e^{-\sqrt{m}} \right\} \nonumber
    \end{align}
    is at least $1-\left(\varepsilon/2  \right)^{n}$.
\end{mylem}
\begin{proof}
    Let $\varepsilon'$ will be chosen later and be small enough. Let $m \ge \left( \varepsilon' z\right)^{-4} \ge 10$. For any $s \in [z,\varepsilon'k]$ and $i \le m$ denote
    \begin{align}
        \gamma_{i}(s):=\left| \frac{1}{\left\lfloor n/m \right\rfloor} \sum_{w \in S_{i}}\textsf{E}_{\xi}\exp(2 \pi \textbf{i}X_{w}\xi s)  \right|, \ \ f_{i}(s):=\psi_{K_{2}}(\gamma_{i}(s)), \ \text{ and }\nonumber
    \end{align}
    \begin{align}
        f(s):=\prod_{i=1}^{m}f_{i}(s).\nonumber
    \end{align}
    Recall the definition of $\psi_{K_{2}}$, we have $f_{i}(s)=\gamma_{i}(s)$ whenever $\gamma_{i}(s) \ge 1/K_{2}$. Note that for complex number $z_{1},\dots,z_{N}$ with $|z_{i}| \le 1$ their average $v:=\sum_{i=1}^{N}{z_{i}}/N$ has modulus $1-\alpha>0$, then we have 
    \begin{align}
        N(1-\alpha) \le \sum_{i=1}^{n}\mathrm{Re}\left\langle z_{i},v \right\rangle \le N.\nonumber
    \end{align}
    Thus, using Markov's inequality, we have at least $N/2+1$ indices $i$ such that $\mathrm{Re}\left\langle z_{i},v \right\rangle \ge 1-4\alpha$. Furthermore, there exists an index $j$ such that there are at least $N/2$ indices $i$ with $\mathrm{Re}\left\langle z_{i},\bar{z_{j}} \right\rangle \ge 1-16\alpha$. Thus, the event $\left\{ f_{i}(s) \ge 1-\frac{2}{\sqrt{m}}  \right\}$ is contained in the event
    \begin{equation}
    \begin{aligned}
        \bigg\{ \exists w' \in S_{i}:\textsf{E}_{\xi}\cos{ \left( 2\pi \xi s (X_{w}-X_{w}') \right) } \ge 1-\frac{32}{\sqrt{m}}\\ \text{ for at least } \frac{n}{2m} \text{ indices } w \in S_{i}\setminus \{ w'\}  \bigg\}. \nonumber
    \end{aligned}
    \end{equation}
    We now have 
    \begin{equation}
        \begin{aligned}
             \textsf{P} & \left( f_{i}(s) \ge 1-\frac{2}{\sqrt{m}} \right)\\
            & \le \frac{n}{m}2^{n/m}\max_{w'\in S_{i},F\subset S_{i}\setminus \{w' \} |F| \ge n/(2m) }{\textsf{P}\left( \forall w \in F:\textsf{E}_{\xi}\mathrm{dist}(\xi s(X_{w}-X_{w'}),\mathbb{Z}) \le \frac{2}{m^{1/4}} \right)   }.\nonumber
        \end{aligned}
    \end{equation}
    Note that if we fix $X_{w'}$, we have $X_{w}-X_{w'} \in \frac{1}{k}\mathbb{Z}\cap [h_{1},h_{2}]$, where $h_{2}-h_{1} \ge 2$. Applying Lemma \ref{Distance from integral} for $s \in [z,\varepsilon'k]$, $\varepsilon=\frac{2}{m^{1/4}}$ and $\xi$ satisfying \ref{Condition}, we obtain
    \begin{align}
        \textsf{P}\left( \textsf{E}_{\xi}\mathrm{dist}(\xi s(X_{w}-X_{w'}),\mathbb{Z}) \le \frac{2}{m^{1/4}} \right) \le C_{\xi}\varepsilon'/z.\nonumber
    \end{align}
    Furthermore, we have 
    \begin{align}
        \textsf{P}\left( f_{i}(s) \ge 1-\frac{2}{\sqrt{m}} \right) \le \frac{n}{m}\left( \frac{C_{\xi}\varepsilon'}{z} \right)^{n/(2m)}.\nonumber
    \end{align}
    Using this estimate and recall the definition of the $\psi_{K_{2}}$, we have 
    \begin{equation}
    \begin{aligned}
        \textsf{P}\left(  f(s) \ge (1-\frac{2}{\sqrt{m}})^{3m/4} \right)
        & \le \textsf{P}\left( f_{i}(s) \ge 1-\frac{2}{\sqrt{m}} \text{ for at least } m/4 \text{ indices } i \right)\\
        & \le \left( \frac{16n}{m}\right)^{m/4}\left( \frac{C_{\xi}\varepsilon'}{z} \right)^{n/8}.\nonumber
    \end{aligned}
    \end{equation}
    Finally, we discrete the interval $[z,\varepsilon'k]$ and recall for $X \in \Lambda_{n}$, we have$\Vert X\Vert_{\infty} \le  \textbf{g}(n) \le 2^{n}$. Then by the Lemma \ref{Lipschitzness of the product}, we have $f(s)$ is $C_{\ref{Lipschitzness of the product}}2^{n}m$-Lipschitz.\par 
    Set 
    \begin{align}
         \beta:=(1-\frac{2}{\sqrt{m}})^{3m/4}\left( C_{\ref{Lipschitzness of the product}}2^{n}m \right)^{-1}   \ \text{ and }\ T:=[z,\varepsilon'k]\cap \mathbb{Z}.\nonumber
    \end{align}
    As the last step, we complete the proof of this lemma by 
    \begin{equation}
    \begin{aligned}
        \textsf{P}&\left( \forall s \in [z,\varepsilon'k]:f(s)\le \left( 1-\frac{2}{\sqrt{m}} \right)^{3m/4}  \right) \\
        & \ge \textsf{P}\left( \forall s \in T:f(s) \le (1-\frac{2}{\sqrt{m}})^{3m/4}  \right)\\
        & \ge 1-\left( \frac{16n}{m}\right)^{m/4}\left( \frac{C_{\xi}\varepsilon'}{z} \right)^{n/8} \ge 1-\left( \varepsilon/2 \right)^{n},\nonumber
    \end{aligned}
    \end{equation}
    where $\varepsilon'=c_{\xi}\varepsilon^{8}z$.
\end{proof}
Our last part is to prove that the integration of the product of $\psi_{K_{2}}(\cdot)$ in a small central interval is small. We give the following lemma.
\begin{mylem}\label{Integration in small interval}
    For any $\varepsilon \in (0,1)$, $\rho \in (0,1/4)$ and $\delta \in (0,1/2)$, there exist $n_{\ref{Integration in small interval}}=n_{\ref{Integration in small interval}}(\varepsilon,\delta,\rho,\xi) \in \mathbb{N}$, $C_{\ref{Integration in small interval}}=C_{\ref{Integration in small interval}}(\varepsilon.\delta,\rho,\xi) \ge 1$ and $K_{\ref{Integration in small interval}}=K_{\ref{Integration in small interval}}(\delta,\rho,\xi) \ge 1$ such that the following holds. Let $n \ge n_{\ref{Integration in small interval}}$, $k\ge 1$, $m \in \mathbb{N}$ with $n/m \ge C_{\ref{Integration in small interval}}$ and $m\ge 2$. Let $X=(X_{1},\dots,X_{n})$ be a random vector uniformly distributed on $\Lambda_{n}$. Then for every $K_{2} \ge 4$, 
    \begin{equation}
        \begin{aligned}
        \textsf{P}& \left( A_{nm}\sum_{S_{1},\dots,S_{m}}\int_{-\sqrt{m}/C_{\ref{Integration in small interval}}}^{\sqrt{m}/C_{\ref{Integration in small interval}}}\prod_{i=1}^{m}\psi_{K_{2}}\left( \left| \textsf{E}_{\xi}\exp{(2\pi \textbf{i}X_{\eta[S_{i}]} \xi m^{-1/2}s )} \right| \right)\mathrm{d}s  \ge K_{\ref{Integration in small interval}} \right)\\
        & \le  \left(\frac{\varepsilon}{2}\right)^{n}.\nonumber
    \end{aligned}
    \end{equation} 
\end{mylem}
\begin{proof}
    Let $\bar{\varepsilon}=\bar{\varepsilon}(\varepsilon,\xi) $ will be chosen later. Recall the definition of $\mathcal{J}$ in Lemma \ref{Combinatorial estimate for Q1 and Q2}, for $S_{1},\dots,S_{m} \in \mathcal{J}$, denote
    \begin{align}
        \gamma_{i}(s):=\left| \frac{1}{\left\lfloor n/m \right\rfloor} \sum_{w \in S_{i}}\textsf{E}_{\xi}\exp(2 \pi \textbf{i}X_{w}\xi s)  \right|, \ \ f_{i}(s):=\psi_{K_{2}}(\gamma_{i}(s)), \ \text{ and }\nonumber
    \end{align}
    \begin{align}
        f(s):=\prod_{i=1}^{m}f_{i}(s).\nonumber
    \end{align}
    In the following proof, we decompose $\mathcal{J}$ into two parts and then estimate the corresponding values of $f(s)$ separately. Firstly, we assume that $n \ge n_{\ref{Combinatorial estimate for Q1 and Q2}}$ and $n/m \ge C_{\ref{Combinatorial estimate for Q1 and Q2}}$. We also define 
    \begin{equation}
        \begin{aligned}
         \mathcal{J}':= & \bigg\{ (S_{1},\dots,S_{m}) \in \mathcal{J}: \\
        & \min{\left(|S_{i}\cap Q_{1}|,|S_{i}\cap Q_{2}|\right)} \ge \delta\left\lfloor n/m\right\rfloor/2 \text{ for at least } c_{\ref{Combinatorial estimate for Q1 and Q2}}m \text{ indices } i \bigg\}.\nonumber
    \end{aligned}
    \end{equation}
    Fix a $(S_{1},\dots,S_{m}) \in \mathcal{J}'$ and let $J \subset [m]$ be a subset of cardinality $\left\lceil c_{\ref{Combinatorial estimate for Q1 and Q2}} m \right\rceil$ such that
    \begin{align}
        \forall j \in J: \ \ \min\left( |S_{j}\cap Q_{1}|,|S_{j}\cap Q_{2}|  \right) \ge \delta\left\lfloor n/m \right\rfloor/2.\nonumber
    \end{align}
    Thus, for all $i \in J$, within $S_{i}$, we can find at least $\frac{\delta}{2}\left\lfloor n/m \right\rfloor$ disjoint pairs of indices $(w_{1},w_{2}) \in Q_{1}\times Q_{2}$. Let $T $ be a subset of such pairs with cardinality $\frac{\delta\left\lfloor n/m \right\rfloor}{2}$, we obtain for $s >0$ and $a:=\min_{i \le k}|a_{i}|$,
    \begin{equation}
    \begin{aligned}
         & \textsf{P}\left( \gamma_{i}(s) \ge 1-\frac{\pi a^{2} \rho^{2}\delta s^{2}}{2} \right) \\
         \le &  \textsf{P}( \left| \textsf{E}_{\xi}\exp(2\pi \textbf{i}\xi X_{w_{1}}s) +\textsf{E}_{\xi}\exp(2\pi\xi X_{w_{2}}s ) \right| \ge 2-2\pi a^{2} \rho^{2}s^{2}\\
         & \text{ for at least } \frac{\delta}{4}\left\lfloor n/m \right\rfloor \text{ pairs } (w_{1},w_{2}) \in T  ).\nonumber
    \end{aligned}
    \end{equation}
    We return to consider the probability of the following event:
    \begin{align}
        \left\{  \left| \textsf{E}_{\xi}\exp(2\pi \textbf{i} \xi X_{w_{1}}s )+\textsf{E}_{\xi}\exp( 2 \pi \textbf{i} \xi X_{w_{2}}s)  \right|  \ge 2-2\pi \rho^{2}s^{2} \right\}.\nonumber
    \end{align}
    Consider $w_{1} \in Q_{1}$, $w_{2} \in Q_{2}$ and recall $X \in \Lambda_{n}$, we have $g_{1}$ and $g_{2}$ with $g_{1} \ge h+1$ and $g_{2} \le h-\rho-1$ such that 
    \begin{align}
        X_{w_{1}} \in \frac{1}{k}\mathbb{Z}\cap [h,g_{1}] \  \text{ and } \ X_{w_{2}} \in \frac{1}{k}\mathbb{Z}\cap [g_{2},h-\rho].\nonumber 
    \end{align}
    In the cases $g_{1} \le h+2\bar{\varepsilon}^{-1}$ and $g_{2} \ge h-\rho-2\bar{\varepsilon}^{-1}$, we have $a\rho \le |\xi||X_{w_{1}}-X_{w_{2}}| \le 
    \rho +4/\bar{\varepsilon}$, where $a:=\min_{i \le k}|a_{i}|$. Thus, we obtain, for $s \in (0,\frac{\bar{\varepsilon}}{2\rho \bar{\varepsilon}+8})$,
    \begin{equation}
        \begin{aligned}
            \left| \textsf{E}_{\xi}\exp(2\pi \textbf{i} \xi X_{w_{1}}s )+\textsf{E}_{\xi}\exp( 2 \pi \textbf{i} \xi X_{w_{2}}s) \right| 
            & \le \textsf{E}_{\xi}\left| 1+\exp(2\pi \textbf{i} \xi (X_{w_{1}}-X_{w_{2}})s ) \right|\\
            & \le 2\textsf{E}_{\xi}|\cos(\pi \textbf{i} \xi(X_{w_{1}}-X_{w_{2}})s )|\\
            & \le 2-2\pi a^{2} \rho^{2} s^{2}.\nonumber
        \end{aligned}
    \end{equation}
    In the case $g_{1} \ge h+2\bar{\varepsilon}^{-1}$ or $g_{2} \le h-\rho +2\bar{\varepsilon}^{-1}$. Without loss of generality, we assume that the first inequality holds. Fixing $X_{w_{2}}$, there exist $h_{1}$ and $h_{2}$ with $h_{2}-h_{1} \ge 2\bar{\varepsilon}^{-1}$ such that $X_{w_{1}}-X_{w_{2}} \in \frac{1}{k}\mathbb{Z}\cap [h_{1},h_{2}]$. Thus, applying Lemma \ref{Distance from integral} for $s \le\bar{ \varepsilon}$, we obtain 
    \begin{equation}
        \begin{aligned}
        & \textsf{P}\left( \left| \textsf{E}_{\xi}\exp(2\pi \textbf{i} \xi X_{w_{1}}s )+\textsf{E}_{\xi}\exp( 2 \pi \textbf{i} \xi X_{w_{2}}s) \right|  \ge  2-2\pi a^{2} \rho^{2} s^{2}  \right)\\
        &  \le \textsf{P}\left( \textsf{E}_{\xi}\mathrm{dist}\left( \xi s(X_{w_{1}}-X_{w_{2}}) ,\mathbb{Z} \right) \le s  \right)
        \le C_{\xi}\bar{\varepsilon}.\nonumber
    \end{aligned}
    \end{equation}
    Returning to the estimation of $\gamma_{i}(s)$, based on the above analysis, we further obtain
    \begin{align}
        \textsf{P}\left( \gamma_{i}(s) \ge 1-\frac{\pi a^{2} \rho^{2}\delta s^{2}}{2} \right)  \le \left( C_{\xi}\bar{\varepsilon} \right)^{\delta n/(4m)}.\nonumber
    \end{align}
    Moreover, for any $i \in [m]$, $f_{i}(s) \le 1$ and for any $i \in J$, $f_{i}(s)=\gamma_{i}(s)$ when $\gamma_{i}(s ) \ge 1/K_{2}$. Then set $z:=\min\left( \frac{\bar{\varepsilon}}{2\rho \bar{\varepsilon}+8},\left( \pi a^{2} \rho^{2} \delta  \right)^{-1/2} \right)$, for every $s \in [-z,z]$,
    \begin{equation}
        \begin{aligned}
            \textsf{P} &\left( f(s) \ge \left( 1- \pi a^{2} \rho^{2} \delta s^{2}/2 \right)^{|J|/2}  \right) \\
            & \le \textsf{P}\left( f_{i}(s) \ge 1-\pi a^{2} \rho^{2} \delta s^{2}/2 \text{ for at least } |J|/2 \text{ indices } i\in J  \right)\\
            &  \le \left( C_{\xi}\bar{\varepsilon} \right)^{c_{\delta}n},\nonumber
        \end{aligned}
    \end{equation}
    where $C_{\xi}>0$ depending only on $\xi$ and $c_{\delta} \in (0,1)$ depending only on $\delta$.\par 
    Next, applying the integral Markov inequality (Lemma \ref{Markov integral 1}) to \( f \), and then applying the discrete Markov inequality (Lemma \ref{Markov integral 2}) to the resulting integral, we finally obtain
    \begin{equation}
        \begin{aligned}
            \textsf{P} & \left( A_{nm}\sum_{S_{1},\dots,S_{m} \in \mathcal{J}'}\int_{-z}^{z}f(s)\mathrm{d}s \le \int_{-z}^{z}\left( 1-\frac{\pi a^{2}\rho^{2}\delta s^{2}}{2} \right)^{|J|/2} \mathrm{d}s+2m^{-1/2}  \right)\\
            & \ge 1- 2zm\left( C_{\xi}\bar{\varepsilon} \right)^{c_{\delta}n}.\nonumber
        \end{aligned}
    \end{equation}
    As the last step, recall Lemma \ref{Combinatorial estimate for Q1 and Q2}, we have $|\mathcal{J}'| \ge (1-e^{c_{\ref{Combinatorial estimate for Q1 and Q2}}n})|\mathcal{J}|$, furthermore, we obtain
    \begin{align}
        A_{nm}\sum_{S_{1},\dots,S_{m} \in \mathcal{J}\setminus \mathcal{J}'}\int_{-z}^{z}f(s)\mathrm{d}s \le 2ze^{-c_{\ref{Combinatorial estimate for Q1 and Q2}}n}.\nonumber
    \end{align}
    Combining the analyses from the two preceding parts, we ultimately obtain for $n$ is large enough: 
    \begin{align}
        \textsf{P}\left( A_{nm}\sum_{S_{1},\dots,S_{m} \in \mathcal{J}}\int_{-z}^{z}f(s)\mathrm{d}s  \le Cm^{-1/2} \right) \ge 1-2zm\left( C_{\xi}\bar{\varepsilon} \right)^{c_{\delta}n} \ge 1-\left( \varepsilon/2 \right)^{n},\nonumber
    \end{align}
    where $C>0$ depending on $\xi$, $\delta$ and $\rho$ and we set $\bar{\varepsilon}=(c_{\xi}\varepsilon)^{c_{\delta}}$. The result is then reached by multiplying $s$ by $m^{1/2}$.
\end{proof}
\begin{proof}[\textsf{Proof of Proposition \ref{Anti-concentration of U-degree in Lambda}}]
    We first fix the parameters mentioned in Proposition \ref{Anti-concentration of U-degree in Lambda}: fixing $\delta, \rho \in (0,1/4)$, a growth function $\textbf{g}(\cdot)$ satisfying \eqref{Condition g}, a permutation $\sigma \in \prod_{n}$, a number $h \in \mathbb{R}$, and two disjoint subsets $Q_{1},Q_{2} \subset [n]$ with cardinality $\left\lceil \delta n \right\rceil$. Finally, let $\varepsilon \in (0,1/4)$, $8 \le K_{2} \le 1/\varepsilon$, and $X$ be a random vector uniformly distributed on $\Lambda_{n}$.\par 
    Next, we determine the constants in the proof. Assume that $n$ is large enough. Set $l=l_{\ref{Small product}}(\varepsilon,S_{\xi})$, $z= 1/C_{\ref{Integration in small interval}}(\varepsilon,\delta,\rho,\xi)$, and $\varepsilon'=\varepsilon'(\varepsilon,\xi,z)$ is taken in lemma \ref{Moderately small product}. Finally, for $m ,k \in \mathbb{N}$, set $m \in [C_{\ref{Moderately small product}},n/\max(l, C_{\ref{Integration in small interval}})]$ satisfying $R_{\ref{Small product}}\sqrt{m}e^{-\sqrt{m}} \le 1$ and $1 \le k \le \min(2^{n/l}, \left( K_{2}/8 \right)^{m/8})$.\par 
    As the definition in the proof of Lemma \ref{Integration in small interval}, denote 
    \begin{align}
        f(s):=f_{S_{1},\dots,S_{m}}(s):=\prod_{i=1}^{m}\psi_{K_{2}}\left( \left| \textsf{E}\exp(2\pi \textbf{i} \xi X_{\eta[S_{i}]}m^{-1/2}s ) \right| \right)\nonumber
    \end{align}
    for $S_{1},\dots,S_{m} \in \mathcal{J}$.\par 
    Note that 
    \begin{equation}
        \begin{aligned}
            A_{nm} & \sum_{S_{1},\dots,S_{m}}\int_{-\varepsilon'm^{1/2}k}^{\varepsilon'm^{1/2}k}f(s)\mathrm{d}s \\
            & = A_{nm}\sum_{S_{1},\dots,S_{m}}\int_{-zm^{1/2}}^{zm^{1/2}}f(s)\mathrm{d}s +2A_{nm}\sum_{S_{1},\dots,S_{m}}\int_{z\sqrt{m}}^{\varepsilon'\sqrt{m}k}f(s)\mathrm{d}s.\nonumber
        \end{aligned}
    \end{equation}
    Applying Lemma \ref{Integration in small interval}, the first summand can be bounded by $K_{\ref{Integration in small interval}}$ with probability at least $1-\left( \varepsilon/2 \right)^{n}$.\par 
    Otherwise, for the second summand, we combine Lemma \ref{Small product} and \ref{Moderately small product}. On the one hand, by Lemma \ref{Small product}, the function $f$ is bounded by $\left( K_{2}/4 \right)^{-m/2}$ on $[0,k\sqrt{m}/2]$ except for some set of measures at most $R_{\ref{Small product}}\sqrt{m}$ with probability at least $1-\left( \varepsilon/2 \right)^{n}$. On the other hand, by Lemma \ref{Moderately small product}, the function $f$ is bounded by $e^{-\sqrt{m}}$ on $[z\sqrt{m},\varepsilon'\sqrt{m}k]$ with probability at least $1-\left( \varepsilon/2 \right)^{n}$. Furthermore, with probability at least $1-2\left( \varepsilon/2 \right)^{n}$,
    \begin{align}
        \int_{z\sqrt{m}}^{\varepsilon'k\sqrt{m}}f(s)\mathrm{d}s \le k\sqrt{m}\left( K_{2}/4 \right)^{-m/2}+R_{\ref{Small product}}\sqrt{m}e^{-\sqrt{m}} \le 2.\nonumber
    \end{align}
    Thus, by Lemma \ref{Markov integral 2}, we have 
    \begin{align}
        A_{nm}\sum_{S_{1},\dots,S_{m}}\int_{z\sqrt{m}}^{\varepsilon'k\sqrt{m}}f(s)\mathrm{d}s \le 3.\nonumber
    \end{align}
    Thus, we complete the proof of this proposition.
\end{proof}
\begin{myrem}
The proof of Proposition \ref{Anti-concentration of U-degree in Lambda} shows the main difficulties in analyzing the anti-concentration of the RUD. Because RUD involves averaging over many partitions and integrating products of truncated Fourier coefficients, direct estimates are not feasible. We therefore split the integration interval into a central part and two edge parts.

In the edge intervals, we prove that for most partitions, the product of truncated characteristic functions is exponentially small. This uses combinatorial control of the index sets together with the lattice-type anti-concentration estimate in the Lemma \ref{Distance from integral}, which only requires finite fourth moments and thus extends the U-degree method beyond Bernoulli variables.

In the central region, where the pointwise bounds are too weak, we show that the full integral remains uniformly small. Combining both arguments yields a strong small-ball estimate for RUD, confirming that it retains the stability properties of LCD and U-degree while applying to much broader discrete distributions.
\end{myrem}
\subsection{Unstructured vectors almost have large RUD}\label{Unstructured vectors have almost large U-degree section}
In this section, we first introduce two properties of the RUD and then derive the final result. The proofs of the first two claims rely solely on simple properties of expectation and similar to the proof in Section 4 in \cite{LT_duke}, so we shall omit them.
\begin{mypropo}[Lower bound on the RUD]\label{Lower bound of U degree}
    For any $r, \delta, \rho \in (0,1)$ there exists $C_{\ref{Lower bound of U degree}}=C_{\ref{Lower bound of U degree}}(r,\delta,\rho,\xi)>0$ such that the following holds. Let $K_{2} \ge 2$, $1 \le m\le n/C_{\ref{Lower bound of U degree}}$, $K_{1} \ge C_{\ref{Lower bound of U degree}}$, and let $X \in \mathcal{V}_{n}$. Then,
    \begin{align}
        \mathrm{UD}_{n}^{\xi}(x,m,K_{1},K_{2}) \ge \sqrt{m}.\nonumber
    \end{align}
\end{mypropo}
\begin{mypropo}[Stability of the RUD]\label{Stability of the U degree}
    For any $K_{2} \ge 1$ there exist $c_{\ref{Stability of the U degree}}$ and $c_{\ref{Stability of the U degree}}'$ depending on $K_{2}$ and $\xi$ such that the following holds. Let $K_{1} \ge 1$, $v \in \mathbb{R}^{n}$, $m \le n/2$ and $\mathrm{UD}_{n}^{\xi}(v,m,K_{1},K_{2}) \le c_{\ref{Stability of the U degree}}'k$. Then there are $y \in \frac{1}{k}\mathbb{Z}^{n}$ such that $\Vert v-y \Vert_{\infty} \le \frac{1}{k}$ and satisfying
    \begin{align}
        \mathrm{UD}_{n}^{\xi}(y,m,c_{\ref{Stability of the U degree}}K_{1},K_{2}) \le \mathrm{UD}_{n}^{\xi}(v,m,K_{1},K_{2}) \le \mathrm{UD}_{n}^{\xi}\left( y,m,c_{\ref{Stability of the U degree}}^{-1}K_{1},K_{2} \right).\nonumber
    \end{align}
\end{mypropo}
Finally, we present the main theorem of this section.
\begin{mytheo}\label{Large U degree}
    Let $r, \delta, \rho \in (0,1)$, $s >0 $. $R \ge 1$, and let $K_{3} \ge 1$. Let $\xi$ be a random variable satisfying \eqref{Condition}. Then there exist $n_{\ref{Large U degree}} \in \mathbb{N}$, $C_{\ref{Large U degree}}\ge 1$ and $K_{1} \ge 1$, $K_{2} \ge 4$ depending on $r, \delta,\rho, R,s ,K_{3}$ and $\xi$ such that the following holds. Let $n \ge n_{\ref{Large U degree}}$, $p \le C_{\ref{Large U degree}}^{-1}$, and $s\log n \le pn$. Let $\textbf{g}$ be a growth function satisfying \eqref{Condition g}. Assume that $M_{n}$ is an $n \times n$ random matrix from Theorem \ref{Theorem A}. Then with probability at least $1- \exp(Rpn)$ one has
    \begin{equation}
        \begin{aligned}
            \big\{ & \text{Set of normal vectors to } \mathrm{C}_{2}(M_{n}),\dots,\mathrm{C}_{n}(M_{n})  \big\}\cap \mathcal{V}_{n}(r,\textbf{g},\delta,\rho)\\
             \subset & \big\{ x \in \mathbb{R}^{n}:x_{\left\lfloor rn \right\rfloor}^{*}=1,\mathrm{UD}_{n}^{\xi}(x,m,K_{1},K_{2}) \ge \exp(Rpn) \\
            & \text{ for all } pn/8 \le m \le 8pn \big\}.\nonumber
        \end{aligned}
    \end{equation}
\end{mytheo}
\begin{proof}
    We start by determining the constants. Assume that $n $ is large enough. Fix $R \ge 1$, $r,\delta,\rho \in (0,1)$, $s >0$ and set $b_{0}:=\left\lfloor (2pR)^{-1}\right\rfloor$. Let $K_{2}= 32\exp(16R)$. Note that $\textbf{g}(6\cdot)$ is a growth function for $K_{3}'=K_{3}^{8}$ and $\textbf{g}(6n) \le K_{3}^{n}$.\par
    Furthermore, we denote
    \begin{equation}
      \begin{aligned}
     & C_{\ref{Spectral norm}}:=C_{\ref{Spectral norm}}(3R,\xi), \ \ C_{\ref{all columns concentration}}:= C_{\ref{all columns concentration}}(3R), \ \ c_{\ref{Stability of the U degree}}':=c_{\ref{Stability of the U degree}}'(K_{2},\xi),\\
    & c_{\ref{Stability of the U degree}}:=c_{\ref{Stability of the U degree}}(K_{2},\xi), \ \ C_{\ref{Size of Lambda approximation}}:= C_{\ref{Size of Lambda approximation}}(K_{3}').\nonumber
     \end{aligned}
    \end{equation}
    Next, we set $K_{1}$ is large enough, $pn$ is large enough, and $p $ is small enough such that the statement used in the proof below for our $K_{1}$, $p$ and $n$.\par 
    Finally, let $\varepsilon \le 1/K_{2}$ will be chosen later. For convenience, denote
    \begin{align}
        \mathrm{UD}_{n}(X):=\min_{pn/8 \le m \le 8pn}\mathrm{UD}_{n}^{\xi}(X,m,K_{1},K_{2}).\nonumber
    \end{align}
    Set $H_{1}^{\perp}$ is the set of normal vectors for $\mathrm{C}_{2}(M_{n}),\dots,\mathrm{C}_{n}(M_{n})$. Note that the conclusion is trivially true if $\mathcal{V}_{n} \cap H_{1}^{\perp}=\emptyset$; therefore, we assume $\mathcal{V}_{n}\cap H_{1}^{\perp} \ne \emptyset$. Thus, to prove this theorem it is sufficient to show that
    \begin{align}
        \textsf{P}\left( \exists X \in \mathcal{V}_{n}\cap H_{1}^{\perp}:\mathrm{UD}_{n}(X) \le \exp(Rpn )  \right) \le \exp(-Rpn).\nonumber
    \end{align}
    Applying Lemma \ref{Lower bound of U degree} for $X \in \mathcal{V}_{n}$, we have $\mathrm{UD}_{n}(X) \ge \sqrt{pn/8}$. As the first step, we split the interval $[\sqrt{pn/8},\exp(Rpn) ]$. Set $D \in [\sqrt{pn/8},\exp(Rpn)/2 ]$ and denote
    \begin{align}
        S_{D}:=\left\{ x \in \mathcal{V}_{n}:\mathrm{UD}_{n}(x) \in [D,2D]  \right\}.\nonumber
    \end{align}
    We return to prove that 
    \begin{align}
        \textsf{P}\left( \exists X \in S_{D}\cap H_{1}^{\perp} \right) \le \exp(-2Rpn).\nonumber
    \end{align}
    As the second step, we now give some definitions for events. We say a subset $I \subset [n]$ is admissible if $1 \notin I$ and $|I| \ge n-b_{0}-1$. Then, the integer number $B_{i}$ for $\mathrm{C}_{i}(M_{n})$ is defined by
    \begin{align}
        B_{i}=\left| \{ j\in [n]: \eta_{ij}=b  \} \right|.\nonumber
    \end{align}
    Furthermore, we consider the events $\mathcal{E}_{I}$ for $I$ is admissible are denote by 
    \begin{align}
        \mathcal{E}_{I}:=\big\{ \forall i \in I:  B_{i} \in [pn/8,8pn] \text{ and } \forall i \notin I:B_{i} \notin [pn/8,8pn] \big\}.\nonumber
    \end{align}
    At the same time, we also need to denot
    \begin{align}
        \mathcal{E}_{0}:=\left\{ \Vert M_{n}-\textsf{E}M_{n}\Vert_{2} \le C_{\ref{Spectral norm}}\sqrt{pn}   \right\}.\nonumber
    \end{align}
    Applying Lemma \ref{Spectral norm}, we have $\textsf{P}\left( \mathcal{E}_{0} \right) \ge 1-\exp(-3Rpn)$. By Lemma \ref{all columns concentration}
    \begin{align}
        \textsf{P}\left(  \bigcup_{I}\mathcal{E}_{I} \right) \ge 1- \exp(-n/C_{\ref{all columns concentration}}) \ge 1-\exp(3Rpn).\nonumber
    \end{align}
    Denote by $\ell $ the collection of all admissible $I$ satisfying $2 \textsf{P}\left( \mathcal{E}_{I}\cap \mathcal{E}_{0} \right) \ge \textsf{P}\left( \mathcal{E}_{I} \right)$.By this definition, we have
    \begin{align}
        \textsf{P}\left( \bigcup_{I \in \ell}\mathcal{E}_{I}  \right) \ge 1-\exp(-3Rpn)-2\textsf{P}\left(\mathcal{E}_{0}^{c} \right) \ge 1- 3\exp(-3Rpn).\nonumber
    \end{align}
    Furthermore,
    \begin{equation}
        \begin{aligned}
        \textsf{P}\left( \exists X \in S_{D}\cap H_{1}^{\perp} \right) & \le \sum_{I \in \ell}\textsf{P}\left(\{ \exists X \in S_{D}\cap H_{1}^{\perp}\} \cap \mathcal{E}_{I}\cap\mathcal{E}_{0} \right)+4\exp(-3Rpn)\\
        & \le  \sum_{I \in \ell}\textsf{P}\left( \exists X \in S_{D}\cap H_{1}^{\perp}\big| \mathcal{E}_{I}\cap\mathcal{E}_{0} \right)\textsf{P}\left(  \mathcal{E}_{I}\cap\mathcal{E}_{0} \right)+4\exp(-3Rpn).\nonumber
    \end{aligned}
    \end{equation}
    Combining $\sum_{I \in \ell}\textsf{P}\left( \mathcal{E}_{I}\cap\mathcal{E}_{0} \right) \le 1$, it is sufficient to show that for all $I \in \ell$
    \begin{align}
        \textsf{P}\left( \exists X \in S_{D}\cap H_{1}^{\perp}\big| \mathcal{E}_{I}\cap\mathcal{E}_{0} \right) \le \exp(-3Rpn).\nonumber
    \end{align}
    As the third step, for all $I \in \ell$, denote by $M_{I}$ the $|I|\times n$ matrix obtained by transposing columns $\mathrm{C}_{i}(M_{n})$, $i \in I$ and $M_{I}^{(0)}=\textsf{E}M_{I}=(p\textsf{E}\xi+b)\textbf{1}^{T}\textbf{1}$.\par 
    We now denote $\mathcal{E}_{D,I}$ by 
    \begin{align}
        \mathcal{E}_{D,I}:= \{ \exists X\in S_{D}\cap H_{1}^{\perp} \} \cap \mathcal{E}_{I}\cap \mathcal{E}_{0}.\nonumber
    \end{align}
    Set $k:=\left\lceil 2D/c_{\ref{Stability of the U degree}}' \right\rceil$ and $\textbf{m} : \mathcal{E}_{D,I} \to [pn/8,8pn]$ be a random integer such that
    \begin{align}
        \mathrm{UD}_{n}^{\xi}(X,\textbf{m},K_{1},K_{2}) \in [D,2D] \ \text{ everywhere on } \mathcal{E}_{D,I}.\nonumber
    \end{align}
    Applying Proposition \ref{Stability of the U degree}, there exist $Y:\mathcal{E}_{D,I} \to \frac{1}{k}\mathbb{Z}^{n}$ such that
    \begin{itemize}
        \item $\Vert Y- X\Vert_{\infty} \le 1/k$ on $\mathcal{E}_{D,I}$.
        \item $\mathrm{UD}_{n}^{\xi}(Y,\textbf{m},c_{\ref{Stability of the U degree}}K_{1},K_{2}) \le 2D$ on $\mathcal{E}_{D,I}$.
        \item $\mathrm{UD}_{n}^{\xi}(Y,m,c_{\ref{Stability of the U degree}}^{-1}K_{1},K_{2}) \ge D$ for all $m \in [pn/8,8pn]$.
    \end{itemize}
    It imply that everywhere on $\mathcal{E}_{D,I}$,
    \begin{align}
        \left\Vert \left( M_{I}-M_{I}^{(0)}\right)\left( Y-X \right)  \right\Vert_{2} \le C_{\ref{Spectral norm}}\sqrt{p}n/k.\nonumber  
    \end{align}
    Note that $M_{I}^{(0)}(Y-X)=(p\textsf{E}\xi+b)(\sum_{i=1}^{n}(Y_{i}-X_{i}))\textbf{1}_{I}$, then there exists random number $\textbf{z}:\mathcal{E}_{D,I} \to [-(|b|+1)n/k,(1+|b|)n/k]\cap \frac{\sqrt{pn}}{k}\mathbb{Z}^{n}$ such that 
    \begin{align}
        \Vert M_{I}(Y-\textbf{z}\textbf{1}_{I})\Vert_{2} \le C_{\ref{Spectral norm}}\sqrt{p}n/k.\nonumber
    \end{align}
    Let $\Lambda$ be a subset of
    \begin{align}
\bigcup_{t = \left\lfloor -4\textbf{g}(6n)/\rho \right\rfloor}^{\left\lfloor 4\textbf{g}(6n)/\rho \right\rfloor}\bigcup_{\sigma
         \in \overline{\prod}_{n}} \bigcup_{|Q_{1}|,|Q_{2}| =\left\lceil \delta n \right\rceil } \Lambda_{n}\left(k,\textbf{g}(6\cdot),Q_{1},Q_{2},\rho/4,\sigma,\rho t/4  \right).\nonumber
    \end{align}
    Applying Lemma \ref{Lambda approximation} and $2\textsf{P}\left( \mathcal{E}_{I}\cap \mathcal{E}_{0} \right) \ge \textsf{P}\left( \mathcal{E}_{I} \right)$, we have $Y \in \Lambda$ on $\mathcal{E}_{D,I}$.\par 
    Now, we can obtain 
    \begin{equation}
    \begin{aligned}
        \textsf{P}\left( \mathcal{E}_{D,I}\big|\mathcal{E}_{I}\cap\mathcal{E}_{0} \right) 
        \le & 2\textsf{P}\big( \exists \textbf{z} \in [-(|b|+1)n/k,(|b|+1)n/k]\cap \frac{\sqrt{pn}}{k}\mathbb{Z} \\
        & \text{ and } y \in \Lambda: \Vert M_{I}(Y-z\textbf{1}_{I})\Vert_{2} \le 2C_{\ref{Spectral norm}}\sqrt{p}n/k \big| \mathcal{E}_{I}  \big)\\
         \le & C_{b}|\Lambda|\sqrt{n/p}\max_{z\in \frac{\sqrt{pn}}{k}\mathbb{Z}}\max_{y \in \Lambda}\textsf{P}\left( \Vert M_{I}(Y-z\textbf{1}_{I})\Vert_{2} \le 2C_{\ref{Spectral norm}}\sqrt{p}n/k \big| \mathcal{E}_{I} \right).\nonumber
    \end{aligned}
    \end{equation}
    Applying Theorem \ref{Small ball prob via UD} and Lemma \ref{tensorization},
    \begin{align}
        \textsf{P}\left( \Vert M_{I}(Y-z\textbf{1}_{I})\Vert_{2} \le 2C_{\ref{Spectral norm}}\sqrt{p}n/k \big| \mathcal{E}_{I} \right) \le \left( C/D \right)^{|I|}.\nonumber
    \end{align}
    At the same time, applying $\textbf{g}(6n) \le K_{3}^{n}$, Lemma \ref{Size of Lambda approximation} and Proposition \ref{Anti-concentration of U-degree in Lambda}, we obtain
    \begin{align}
        |\Lambda| \le C(pn/\rho)\varepsilon^{n}(4K_{3})^{n}\left( C_{\ref{Size of Lambda approximation}}k \right)^{n} \le \left( C'\varepsilon k \right)^{n}.\nonumber
    \end{align}
    As the last step, we complete the proof by 
    \begin{equation}
        \begin{aligned}
            \textsf{P}\left( \mathcal{E}_{D,I} \big| \mathcal{E}_{I}\cap\mathcal{E}_{0} \right) 
            & \le \left( C'\varepsilon k \right)^{n}\cdot\left(C/D \right)^{|I|} \\
            & \le \varepsilon^{n}C^{n}N^{1+(2pR)^{-1}}\\
            & \le \exp(-3Rn),\nonumber
        \end{aligned}
    \end{equation}
    where $\varepsilon$ is small enough.
\end{proof}

\section{Structured vectors}\label{Structured vectors section}
In this section, we will introduce the complement of unstructured vectors for $\frac{C\log{n}}{n} \le p \le C^{-1}$. Firstly, recalling the definition of Steep vectors in Subsection \ref{Decomposition of Rn}, we fix the choice of $C_{0}$ from Lemma \ref{Levy concentration} , $C_{1}$ and $\gamma$ as follows.
 \begin{align}\label{lambda}
     C_{1}:= \frac{a'}{2a''} , \gamma=\min(2C_{\ref{Sum concentration}}/a,2C_{\ref{Sum concentration}}/\bar{a}) \text{ and } C_{2}= \frac{2a''(|b|+a'')}{|b|\bar{a}},
 \end{align}
 where 
 \begin{align}
     a:= \min_{i}|a_{i}|, \  &  \bar{a}:=\min_{i \ne j}|a_{i}-a_{j}| \nonumber \\
     a'=\min_{r \ne 0}\left\{ |r|:\textsf{P}\left( \eta =r \right)>0  \right\}&  \text{ and } a'':=\max_{r\in \mathbb{R}}\left\{|r|: \textsf{P}(\eta = r) >0 \right\}. \nonumber
 \end{align}
 The following lemma provides a simple estimate for the Euclidean norm bound of steep vectors, similar to Lemma 6.4 in \cite{LT_duke}.
 \begin{mylem}\label{Steep Euclidean norm}
     Let $n$ be large enough and $200\log{n}/n \le p \le 0.001$. Consider the steep vectors $x \in \mathcal{T}_{1j}$, $1 \le j \le s_{0}+1 $, $y \in \mathcal{T}_{2} $ , $z \in \mathcal{T}_{3}$ and $w \in \mathcal{T}^{c}$. We have
     \begin{align}
         \frac{\Vert x \Vert_{2}}{x_{n_{j-1}}^{*}} \le \frac{C_{\ref{Steep Euclidean norm}}^{(1)}n^{2}(pn)^{2}}{(64p)^{\kappa}}, \ \frac{\Vert y\Vert_{2}}{y_{n_{s_{0}}+1}^{*}} \le \frac{C_{\ref{Steep Euclidean norm}}^{(2)}n^{2}(pn)^{3}}{(64p)^{\kappa}}, \nonumber  
     \end{align}
     \begin{align}
         \frac{\Vert z\Vert_{2}}{z_{n_{s_{0}}+2}^{*}} \le \frac{C_{\ref{Steep Euclidean norm}}^{(2)}C_{\tau}n^{2}(pn)^{3.5}}{(64p)^{\kappa}}, \ \frac{\Vert w\Vert_{2} }{w_{n_{s_{0}}+3}^{*}} \le \frac{C_{\ref{Steep Euclidean norm}}^{(2)}C_{\tau}^{2}n^{2}(pn)^{4}}{(64p)^{\kappa}},\nonumber
     \end{align}
     where $C_{\ref{Steep Euclidean norm}}^{(1)}$ and $C_{\ref{Steep Euclidean norm}}^{(2)} >0$ depending on $\gamma$ and $C_{1}$.
 \end{mylem}
Next, we will divide the complement of the unstructured vector into three parts and complete the proof separately for each.
\subsection{\texorpdfstring{$\mathcal{T}_{0}$ and $\mathcal{T}_{1}$}{T0 and T1}}\label{T0 and T1 section}
In this subsection, we focus on the lower bound of $\Vert Mx\Vert_{2}$ for the vectors from $\mathcal{T}_{0}\cup \mathcal{T}_{1}$. Now, we begin with the following combinatorial lemma for random matrices. 
\begin{mylem}\label{combinatorial lemma}
    There exist a absolute constant $c_{\ref{combinatorial lemma}}$ such that the following holds. Let $n \ge 30$, and $0 <p < c_{\ref{combinatorial lemma}}$ satisfy $pn \ge 200 \log{n}$. Let $m, l=l(m) \in \mathbb{N}^{+}$ be such that 
    \begin{align}
        m \ge 3 , \ \text{ } \ lm \le 1/(64p) \  \text{ and } \ l \le \frac{pn}{4\log{(1/(pm))}}.\nonumber
    \end{align}
    Let $M$ be an $n \times n$ random matrix from Theorem \ref{Theorem A}, which $b \in \mathbb{R}$ and $q\le 1$ is large enough. By $\mathcal{E}_{col}(m,l)$ denote the event that for any choice of two disjoint subsets of $[n]$, $J_{1}$ and $J_{2}$ with $|J_{1}|=m$ and $|J_{2}|=lm-m$ there exist two rows of $M$ such that one of this row is all $b$ in the index of $J_{1} \cup J_{2}$, and the other of this row with exactly one $|\delta_{ij}\xi_{ij}| \ge a:=\min_{i \le L}{|a_{i}|}$ among components indexed by $J_{1}$ and all $b$ in other index of $J_{1} \cup J_{2}$. Then $\textsf{P}\left( \mathcal{E}_{col}  \right) \ge 1- \exp{(-2pn)}$. 
\end{mylem}
\begin{proof}
   Fixing two disjoint sets $J_{1} , J_{2} \subset [n]$ satisfies the assumption of this lemma.\par
    The probability of fixing two rows of $M$ satisfying the assumption equals: \begin{align}
        P= 2mp(1-p)^{2lm-1} \ge  2mp\exp{(-2plm)} \ge 31pm/16\nonumber
    \end{align}
    We choose a pair of two disjoint rows from $M$, by the independent of the pairs, we have the probability of there don't exist two rows satisfies the assumption is at most
    \begin{align}
         (1-P)^{n/2}  \le \exp{\left(-mpn\exp{(-2lmp)}   \right)} \le \exp{(-31mpn/32)},\nonumber
    \end{align}
    Thus, by choosing two disjoint subsets $J_{1}$ and $J_{2}$, we have
    \begin{align}
        \textsf{P}\left( \mathcal{E}_{col}^{c}  \right) \le \binom{n}{lm-m}\binom{n-lm+m}{m}e^{-31pmn/32} \le \left( \frac{3n}{lm } \right)^{lm} (2l)^{m} \exp{\left( -31pmn/32 \right)}.\nonumber
    \end{align}
   For $l \le \frac{pn}{4\log{(1/(pm))}}$, $p \le c_{\ref{combinatorial lemma}}$ and $m \le 5$ is small enough, we have:
   \begin{align}
       \left(   \frac{3n}{lm} \right)^{lm} \le \left( \frac{12\log{(1/(pm))}}{pm}\right)^{\frac{pmn}{4\log{\frac{1}{pm}}}} \le e^{7mpn/24} \text{ and } (2l)^{m} \le e^{pmn/100}.\nonumber
   \end{align}
   Furthermore, we have 
   \begin{align}
       \textsf{P}\left( \mathcal{E}_{col}^{c} \right) \le \exp{(-31/32+7/24+1/100)mpn} \le e^{-2pn}.\nonumber
   \end{align}
   Otherwise, for $1/(64p)\ge m \ge 5$, we also have 
   \begin{align}
       \textsf{P}\left( \mathcal{E}_{col}^{c} \right) \le e^{-2pn}.\nonumber
   \end{align}
    We now complete the proof.
\end{proof}
Next, we can give the first proposition of the Steep vectors.
\begin{mypropo}\label{T0 and T1}
    Let $n \in \mathbb{N}^{+}$ be large enough and $p < c_{\ref{combinatorial lemma}}$ with $pn \ge 200 \log{n}$. Then 
    \begin{equation}
    \begin{aligned}
        & \textsf{P}\left( \exists x \in \mathcal{T}_{0}\cup \mathcal{T}_{1} : \Vert Mx \Vert_{2} \le c_{\ref{T0 and T1} }\frac{(64p )^{\kappa}}{n^{2}(pn)^{2}}\Vert x\Vert_{2}   \right) \\
        & \le \left(  1+o_{n}(1) \right)n\textsf{P}\left( \eta = 0 \right)^{n} + \frac{(1+o_{n}(1))n(n-1)}{2}\textsf{P}\left( \eta'=\eta \right)^{n},\nonumber
    \end{aligned}
    \end{equation}
    where $M$ from Theorem \ref{Theorem A} and $c_{\ref{T0 and T1}}>0$ depending on $\xi$.
\end{mypropo}
\begin{proof}
    We begin with the definitions of some events. For $M=\left(\eta_{ij}  \right)_{i\le n,j \le n}$ from Theorem \ref{Theorem A}, let $\mathcal{E}_{0}$ be the event that there not exists zero columns, which implies $\textsf{P}\left( \mathcal{E}_{0} \right) \ge 1-n\textsf{P}\left( \eta=0 \right)^{n}$. Below we define $\mathcal{E}_{1}$ as the random set of matrices $M$ satisfying one of the following conditions.
    \begin{itemize}
        \item there are no two columns in $M$ satisfying if two rows have equal entries in one column, then the corresponding rows must also have equal entries in the other column.
        \item there are two columns in $M$ satisfying the whenever an entry in one column equals to $b$, the corresponding entry in the other column also equals to $b$; and the new column vectors obtained by removing all entries equal to b satisfying if two rows have equal entries in one column, then the corresponding rows must also have equal entries in the other column.
    \end{itemize}
    Then for each $i \in [L+1]$ and $j \in [n]$, set
    \begin{align}
        S_{i}^{(j)}:=\{ t \in [n]: \eta_{tj}=a_{i}\mathbf{1}_{\{i \ne L+1\} }+b  \}.\nonumber
    \end{align}
    On the consider $\mathcal{E}_{1}^{c}$, there exist $j_{1}$ and $j_{2}$ such that the following holds. Let $T_{L}:= \left\{ \sigma \in \prod_{L+1}: \sigma(L+1)\ne L+1  \right\}$, we can conclude that there exists a permutation $ \sigma \in T_{L}$ such that
    \begin{align}
        S_{i}^{(j_{1})}=S_{\sigma(i)}^{(j_{2})} \ \  \ \ \  \text{ for each   } \ i \in [L+1].\nonumber 
    \end{align}
    Let $q_{i}=pp_{i}$ for $i \le L$ and $q_{L+1}:=1-p$. Note that for each $j_{1}$ and $j_{2}$:
    \begin{equation}
        \begin{aligned}
            \textsf{P}\left( \exists \sigma \in \prod_{L+1} : S_{i}^{(j_{1})}=S_{i}^{j(2)} \right)
            & \le \sum_{\sigma \in T_{L}}\textsf{P}\left( S_{i}^{(j_{1})}=S_{\sigma(i)}^{(j_{2})} \right) \\
            & \le \sum_{\sigma \in T_{L}}\sum_{S_{1},\dots,S_{L+1} \subset [n]}\prod_{i=1}^{L+1}\left( q_{i}q_{\sigma(i)} \right)^{|S_{i}|}\\
            & \le \sum_{\sigma \in T_{L}}\left( \sum_{i=1}^{L+1}{q_{i}q_{\sigma(i)}} \right)^{n}\\
            & \le (1+o_{n}(1))\textsf{P}\left(\eta'=\eta \right)^{n}.\nonumber
        \end{aligned}
    \end{equation}
    Thus, we have
    \begin{align}
        \textsf{P}\left( \mathcal{E}_{1}^{c} \right) \le (1+o(1))\binom{n}{2}\textsf{P}\left( \eta'=\eta \right)^{n}.\nonumber
    \end{align}
    Finally, We set $\mathcal{E}_{j}= \mathcal{E}_{col}(l_{0},n_{j-1})$ as the event from Lemma \ref{combinatorial lemma} for every $2 \le j \le s_{0}+1$, $\textsf{P}\left(  \mathcal{E}_{j} \right) \ge 1-e^{-2pn}$. \par 
    Next, recall the definition of $\sigma_{x}$. For any $x \in \mathcal{T}_{0} \cup \mathcal{T}_{1}$, denote $m=m_{1}=1$ and $m_{2}=2$ if $x \in \mathcal{T}_{0}$. Let $m=m_{1}=n_{j-1}$ and $m_{2}=n_{j}$ if $x \in \mathcal{T}_{1j}$ for some $1 \le j \le s_{0}+1$. Set
    \begin{align}
        J_{1}=J_{1}(x)=\sigma_{x}([m]), \ J_{2}=J_{2}(x)=\sigma_{x}([m_{2}-1]\setminus[m]) \text{ and } J= (J_{1}\cup J_{2})^{c}.\nonumber
    \end{align}
    We also set ``the overall event" as 
    \begin{align}
        \mathcal{E}= \mathcal{E}_{\mathrm{sum}} \cap\bigcap_{j=0}^{s_{0}+1}\mathcal{E}_{j},\nonumber
    \end{align}
    where $\mathcal{E}_{\mathrm{sum}}$ be introduced in Lemma \ref{Sum concentration}.\par 
    Conditioned on $\mathcal{E}$, for $m \ge 3$, there exist $i_{1}$-row of $M$ and $i_{2}$-row of $M$ such that one of two rows is all $b$ in $J_{1}\cup J_{2}$ and other of two rows is exactly one $\delta_{ij}|\xi_{ij}| \ge a$ in $J_{1}$ and all $b $ in $J_{2}$. Without loss of generality, assume that the $i_{1}$-row is all $b$ in $J_{1} \cup J_{2}$ and set $j_{2}=j(i_{2}) \in J_{1}$ such that $\delta_{i_{2}j_{2}}|\xi_{i_{2}j_{2}}| \ge a$.\par 
    We now have 
    \begin{align}
        \Vert Mx\Vert_{2} \ge \left| \left\langle \mathrm{R}_{i_{2}}(M)-\mathrm{R}_{i_{1}}(M),x \right\rangle  \right|/\sqrt{2} \ge \delta_{i_{2}j_{2}}|\xi_{i_{2}j_{2}}||x_{j_{2}}|-x_{m_{2}}^{*}\sum_{j=1}^{n}{|\zeta_{j}|},\nonumber
    \end{align}
    where $\zeta_{j}:= |\delta_{i_{2}j}\xi_{i_{2}j}-\delta_{i_{1}j}\xi_{i_{1}j}|$. Note that conditioned on $\mathcal{E}_{\mathrm{sum}}$ we have $\sum_{j=1}^{n}|\zeta_{j}| \le C_{\ref{Sum concentration}}d$.\par 
    Thus, conditioned on $\mathcal{E}$, we have for all $x \in \bigcup_{j=2}^{s_{0}+1} \mathcal{T}_{1j}$,
    \begin{align}
        \Vert Mx\Vert_{2} \ge ax_{m}^{*}-\frac{C_{\ref{Sum concentration}}x_{m}^{*}}{\gamma} \ge ax_{m}^{*}/2.\nonumber
    \end{align}
    In the case $m=1$: conditioned on $\mathcal{E}_{0}$, recall $a'=\min_{r \ne 0}\left\{ |r|:\textsf{P}\left( \eta =r \right)>0  \right\}$ and $a'':=\max_{r\in \mathbb{R}}\left\{|r|: \textsf{P}(\eta = r) >0 \right\}$, there exist $i$ such that  
    \begin{align}
        \Vert Mx\Vert_{2} \ge  \left| \left\langle \mathrm{R}_{i}(M),x \right\rangle  \right| \ge a'x_{1}^{*}-x_{2}^{*}a''n \ge a'x_{1}^{*}/2.\nonumber
    \end{align}
    In the case $m=2$, conditioned on $\mathcal{E}_{1}\cap \mathcal{E}_{sum}$, set $\bar{a}:=\min_{i \ne j}|a_{i}-a_{j}|$, if there exist $i_{1}$, $i_{2}$ such that $\eta_{i_{1}\sigma_{x}(1)}=\eta_{i_{2}\sigma_{x}(1)}$ and $\eta_{i_{2}\sigma_{x}(2)} \ne \eta_{i_{2}\sigma_{x}(2)}$. Then we have 
    \begin{align}
        \Vert Mx\Vert_{2} \ge \bar{a}x_{2}^{*}-x_{n_{1}}^{*}C_{\ref{Sum concentration}}d \ge \bar{a}x_{2}^{*}/2.\nonumber
    \end{align}
    Otherwise, there exist $i_{1}$ and $i_{2}$ such that
    \begin{align}
        \eta_{i_{1}\sigma_{x}(1)}=\eta_{i_{2}\sigma_{x}(2)}=b \text{ and } \eta_{i_{2}\sigma_{x}(1)} \ne \eta_{i_{2}\sigma_{x}(2)} \ne b .\nonumber
    \end{align}
    Set $\eta_{i_{2}\sigma_{x}(1)}=c_{1}$ and $\eta_{i_{2}\sigma_{x}(1)}=c_{2}$, we have
    \begin{equation}
        \begin{aligned}
         & \max\{ |bx_{\sigma_{x}(1)}+ bx_{\sigma_{x}(2)}|, |c_{1}x_{\sigma_{x}(1)}+c_{2}x_{\sigma_{x}(2)}|  \} \\
         & \ge \frac{|c_{1}|}{|b|+|c_{1}|}|bx_{\sigma_{x}(1)}+ bx_{\sigma_{x}(1)}|+ \frac{|b|}{|b|+|c_{1}|}|c_{1}x_{\sigma_{x}(1)}+c_{2}x_{\sigma_{x}(2)}|\\
         & \ge \frac{|b||c_{1}-c_{2}|}{|b|+|c_{1}|}x_{2}^{*} \ge \frac{|b|\bar{a}}{|b|+a''}x_{2}^{*}.\nonumber
    \end{aligned}
    \end{equation}
    Thus, we obtain
    \begin{align}
        \Vert Mx\Vert_{2} \ge \max_{i \in [n]}{ \left| \left\langle \mathrm{R}_{i}(M), x \right\rangle \right| } \ge \frac{|b|\bar{a}}{|b|+a''}x_{2}^{*}-x_{3}^{*}a''n \ge c'x_{2}^{*}/2,\nonumber
    \end{align}
    where $c'=\frac{|b|\bar{a}}{2(|b|+a'')}$.
    
    Note that for $x \in \mathcal{T}_{0}$, $\Vert x\Vert_{2} \le \sqrt{n}x_{1}^{*}$ and for $x \in \mathcal{T}_{1}$, we have 
    \begin{align}
        \Vert x\Vert_{2} \le \frac{C^{(1)}_{\ref{Steep Euclidean norm}}n^{2}(pn)^{2}}{(64p)^{\kappa}}x_{m}^{*},\nonumber
    \end{align}
    by Lemma \ref{Steep Euclidean norm}.\par
    We have for any $x \in \mathcal{T}_{0}\cup \mathcal{T}_{1}$
    \begin{align}
        \Vert Mx\Vert_{2} \ge c\frac{(64p)^{\kappa}}{n^{2}(pn)^{2}}\Vert x\Vert_{2}.\nonumber
    \end{align}
    Finally, we complete the proof of this lemma since 
    \begin{align}
        \textsf{P}\left( \mathcal{E}^{c} \right) \le (1+o_{n}(1))\binom{n}{1}\textsf{P}(\eta =0)^{n}+(1+o_{n}(1))\binom{n}{2}\textsf{P}(\eta'=\eta)^{n}.\nonumber
    \end{align}
\end{proof}
\subsection{\texorpdfstring{$\mathcal{T}_{2}$ and $\mathcal{T}_{3}$}{T2 and T3}}\label{T2 and T3 section}
In this subsection, let us turn to the remaining part of the Steep vectors and begin with the following two lemmas, which are similar to Lemmas 6.6 and 6.7 in \cite{LT_duke}. The first lemma is a combinatorial lemma on the $n/2 \times n $ Bernoulli random matrices.
\begin{mylem}\label{combinatorial lemma of T2 and T3}
    Let $l \ge 1 $ be an integer and $p \in (0,1/2]$ with $lp \le 1/32$. Let $M_{0}$ be a $n/2 \times n$ random matrix with i.i.i. entries that are Bernoulli(p) and $M:=(\delta_{ij}\xi_{ij})_{i,j}$ be $n/2 \times n$ random matrix as a submatrix introduced in Theorem \ref{Theorem A} when $b =0$. Then with probability at least
    \begin{align}
        1-2\binom{n}{l}\exp{\left( -nlp/9 \right)} \nonumber
    \end{align}
    for every $J \subset [n]$ of cardinality $l$ and large enough $q$ one has
    \begin{align}
        lpn/16 \le |I(J,M_{0})| \le 2lpn\nonumber
    \end{align}
    where we denote $I_{J}:=I(J,M_{0})$ by 
    \begin{align}
        I(J,M_{0}):= \left\{ i \le n/2: \left| \mathrm{supp}\left(\mathrm{R}_{i}( M_{0}) \right) \cap J  \right|=1  \text{ and for those }j \text{ with } |\xi_{ij}| \ge a    \right\}.\nonumber
    \end{align}
    Furthermore, let $l = 2\left\lfloor 1/\left( 64p \right) \right\rfloor \le n$, $n$ be large enough and $p \in (1000/n, 0.001)$. Then, denoting 
    \begin{align}
        \mathcal{E}_{card}:=\left\{  M_{0}: \forall J \subset [n] \text{ with } |J|=l \text{ one has } \left| I(J,M_{0}) \right| \in [lpn/16,2lpn]  \right\}.\nonumber
    \end{align}
    We have $\textsf{P}\left( \mathcal{E}_{card} \right) \ge 1 - \exp{(-n/1000)}$.
\end{mylem}
The second lemma is the net argument of $\mathcal{T}_{2} \cup \mathcal{T}_{3}$. We first give the following normalization:
\begin{align}
    \mathcal{T}_{2}^{*}:= \left\{ x \in \mathcal{T}_{2}: x_{n_{s_{0}}+1}^{*}=1 \right\} \text{ and } \mathcal{T}_{3}^{*} := \left\{ x \in \mathcal{T}_{3}: x_{n_{s_{0}}+2}^{*}=1  \right\}.
\end{align}
Recall the definition of $\Vert x \Vert_{\textbf{e}}$, we give the following lemma.
\begin{mylem}\label{Net for T2 and T3}
     Let $n \in \mathbb{N}^{+}$, $p \in (0, 0.001)$ with $d=pn$ be large enough. Let $i \in \{ 2.3 \}$. Then there exists a set $\mathcal{N}_{i}= \mathcal{N}_{i}^{(1)}+\mathcal{N}_{i}^{(2)}$, $\mathcal{N}_{i}^{(1)} \subset \mathbb{R}^{n}$ and $\mathcal{N}_{i}^{(2)} \subset \mathrm{Span}(\textbf{1})$ such that the following holds:
      \begin{itemize}
   \item $\left| \mathcal{N}_{i} \right| \le C_{\ref{Net for T2 and T3}}n^{2} \exp{\left( 2n_{s_{0}+i}\log{d} \right)}$, where $C_{\ref{Net for T2 and T3}}$ depending only on $\xi$.
    \item For every $u \in \mathcal{N}_{i}^{(1)}$ one has $u_{j}^{*}=0$ for all $ j \ge n_{s_{0}+i}$.
    \item For all $x\in \mathcal{T}_{i}^{*}$, there are $u \in \mathcal{N}_{i}^{(1)}$ and $v \in \mathcal{N}_{i}^{(2)}$ scuh that
    \begin{align}
        \Vert x- u \Vert_{\infty} \le \frac{1}{C_{\tau}\sqrt{d}}, \ \ \Vert v \Vert_{\infty} \le \frac{1}{C_{\tau}\sqrt{d}}, \text{and}\nonumber
    \end{align}
    \begin{align}
        \Vert x-u-v\Vert_{\textbf{e}} \le \frac{\sqrt{2n}}{C_{\tau}\sqrt{d}}.\nonumber
    \end{align}
\end{itemize}     
\end{mylem}
\begin{myrem}
   Note that, compared to Lemma 6.7 in \cite{LT_duke}, the change in $\Vert x\Vert_{\infty}$ leads to a difference here; however, this does not affect the final conclusion and we still obtain a result similar to that lemma, namely the one stated above. Likewise, in the subsequent net estimates for the $\mathcal{R}$-vectors, a comparable conclusion can also be reached.
\end{myrem}
We also need the anti-concentration inequality of vectors in $\mathcal{T}_{2}\cup \mathcal{T}_{3}$, which is similar to the individual probability in \cite{LT_duke}. Thus, we will provide a concise proof that focuses on highlighting the differences while omitting the identical parts. We begin with some definitions.\par
Fix $q_{0} \le n$ and a partition $J_{0}, J_{1}, \dots ,J_{q_{0}}$ of $[n]$. Let $I_{1}, I_{2}, \dots, I_{q_{0}} \subset [n/2]$ and $V=(v_{ij})$ be an $n/2 \times |J_{0}|$ matrix with $0/1$ entries. Let $\mathcal{I}=\left( I_{1},I_{2},\dots, I_{q_{0}} \right)$ and $M^{0}=\left( \delta_{ij}\right)$ be $n/2 \times n$ Bernoulli(p) random matrix. Consider the event:
\begin{align}
    \mathcal{F}\left( \mathcal{I},V \right):= \left\{ M:\forall k \in [q_{0}] \ I\left(J_{k},M \right)=I_{k} \text{ and } M_{J_{0}}^{0}=V  \right\}.\nonumber
\end{align}
Next, denoted by $\textsf{P}_{\mathcal{F}}$ the induced probability measure on $\mathcal{F}(\mathcal{I},V)$, s.t.
\begin{align}
    \textsf{P}_{\mathcal{F}}\left( A \right) := \frac{\textsf{P}\left( A \right)}{\textsf{P}\left( \mathcal{F} \right)}, \ A \subset \mathcal{F}.\nonumber
\end{align}
Note that $\xi_{ij}$ and the $(\delta_{i1}, \dots,\delta_{in})$ remain independent for all $i \le n/2$ andd $j \le n$.\par 
Finally, for $i \le n/2$ and $k \le q_{0}$, define
\begin{align}
    \xi_{k}(i)=\xi_{k}(M,v,i):= \sum_{j \in J_{k}}{\delta_{ij}\xi_{ij}v_{j}},\nonumber
\end{align}
where $M:=(\delta_{ij}\xi_{ij})_{i,j}$ be $n/2 \times n$ random matrix as a submatrix introduced in Theorem \ref{Theorem A} when $b =0 $ and $M^{0}=\left( \delta_{ij} \right)_{i \le n/2,j\le n}$.\par 
We now give our individual probability.
\begin{mylem}\label{Individual Probability}
    There exist constants $C_{\ref{Individual Probability}}, C'_{\ref{Individual Probability}}>1>c_{\ref{Individual Probability}}>0$ depending on $\xi$ with the following property. Let $p \in (0,0.001],d=pn \ge 2$, set $m_{0}=\left\lfloor  1/(64p)\right\rfloor$, let $m_{1}$ and $m_{2}$ be such that
    \begin{align}
        1 \le m_{1} < m_{2} \le n-m_{1}.\nonumber
    \end{align}
    Let $y \in \mathrm{Span}(\textbf{1})$, and assume that $x \in \mathbb{R}^{n}$ such that
    \begin{align}
        x_{m_{1}}^{*} > 2/3 \ \text{ and } \ x_{i}^{*}=0 \ \text{ for every } i >m_{2}.\nonumber
    \end{align}
    Denote $m = \min{(m_{1},m_{0})}$, and consider the event
    \begin{align}
        E_{\omega}\left( x,y \right):= \left\{ M:\Vert M(x+y)-\omega\Vert_{2} \le \sqrt{c_{\ref{Individual Probability}}md}   \right\},\nonumber
    \end{align}
    where $M:=(\delta_{ij}\xi_{ij})_{i,j}$ be $n/2 \times n$ random matrix as a submatrix introduced in Theorem \ref{Theorem A} when $b =0$ and $\omega \in \mathbb{R}^{n/2}$.\par
    Define 
    \begin{align}
        \mathcal{L}_{card}(x,y)=\max_{\omega \in \mathbb{R}^{n/2}}{\textsf{P}\left( E_{\omega}(x,y)\cap \mathcal{E}_{card}  \right)}.\nonumber
    \end{align}
    Then, if $m_{1} \le m_{0}$, we have 
    \begin{align}
        \mathcal{L}_{card}(x,y) \le 2^{-md/40}.\nonumber
    \end{align}
    Otherwise, if $m_{1} \ge C'_{\ref{Individual Probability}}m_{0}$, we have
    \begin{align}
        \mathcal{L}_{card}(x,y) \le \left( \frac{C_{\ref{Individual Probability}}n}{m_{1}d} \right)^{md/40}.\nonumber
    \end{align}
    Here $\mathcal{E}_{card}$ is the event from Lemma \ref{combinatorial lemma of T2 and T3} with $l = 2m$.
\end{mylem}
\begin{proof}
    Recall $d=pn$, fix $f=mp/72=md/(72n)$, $x \in \mathbb{R}^{n}$ and $y \in \mathrm{Span}(\textbf{1})$ satisfying the assumption of this lemma. Denote $q_{0}=m_{1}/m$ and without loss of generality assume that either $q_{0}=1$ or $q_{0}$ is a large integer.\par 
    Let $J_{1}^{(1)},J_{2}^{(1)},\dots,J_{q_{0}}^{(1)}$ be a partition of $\sigma_{x}([m_{1}])$ with cardinality $m$. Similarly, let $J_{1}^{(2)},J_{2}^{(2)},\dots,J_{q_{0}}^{(2)}$ be a partition of $\sigma_{x}([n-m_{1}+1,n])$ with cardinality $m$. Furthermore, let
    \begin{align}
        J_{k}:=J_{k}^{(1)}\cup J_{k}^{(2)} \text{ for each } k \in [q_{0}] \text{ and } J_{0}:=[n]\setminus \bigcup_{k=1}^{q_{0}}{J_{k}}.\nonumber
    \end{align}
    Thus, $J_{0},\dots,J_{q_{0}}$ is a partition of $[n]$. Let $M^{0}:=(\delta_{ij})_{i \le n/2,j\le n}$ be an Bernoulli random matrix. For any $J_{k}^{(1)}$ and $J_{k}^{(2)}$, define the sets $I_{k}^{(1)}$ and $I_{k}^{(2)}$ by 
    \begin{equation}
    \begin{aligned}
        I_{k}^{(1)}:=
        \bigg\{ & i \le n/2:\left| \mathrm{supp}\mathrm{R}_{i}(M^{0}) \cap J_{k}^{(1)}  \right| =1, \text{which }j \text{ with } |\xi_{ij}| \ge a  \\
        & \text{ and } \left| \mathrm{supp}\mathrm{R}_{i}(M^{0}) \cap J_{k}^{(2)}  \right|=0 \bigg\}\nonumber
    \end{aligned}
    \end{equation}
    and 
    \begin{equation}
    \begin{aligned}
        I_{k}^{(1)}:=
        \bigg\{ & i \le n/2:\left| \mathrm{supp}\mathrm{R}_{i}(M^{0}) \cap J_{k}^{(2)}  \right| =1, \text{which }j \text{ with } |\xi_{ij}| \ge a  \\
        & \text{ and } \left| \mathrm{supp}\mathrm{R}_{i}(M^{0}) \cap J_{k}^{(1)}  \right|=0 \bigg\}\nonumber
    \end{aligned}
    \end{equation}
    Let $I_{k}=I_{k}^{(1)} \cup I_{k}^{(2)}$. Note that $|J_{k}| \le 2m\le 1/(32p) $, by the definition of the event $\mathcal{E}_{card}$, we have
    \begin{align}
        |I_{k}| \in [md/8, 4md].\nonumber
    \end{align}
    Fix $\mathcal{I}=(I_{1},\dots,I_{q_{0}})$ and $V$ be an $n/2 \times|J_{0}|$ matrix with $0/1$ entries. Similar to the proof of Lemma 6.11 in \cite{LT_duke}. We have 
    \begin{align}
        \mathcal{L}_{card}(x,y) \le \max_{\omega \in \mathbb{R}^{n}}\textsf{P}\left( E_{\omega}(x,y) | \mathcal{F}(\mathcal{I},V) \right).
    \end{align}
    Fix any class $\mathcal{F}$ and recall the definition of $\textsf{P}_{\mathcal{F}}$:
    \begin{align}
        \textsf{P}_{\mathcal{F}}\left( \cdot \right)= \textsf{P}\left( \cdot | \mathcal{F} \right).\nonumber
    \end{align}
    Denote 
    \begin{align}
        A_{i}=\left\{ k \in [q_{0}]: i \in I_{k} \right\} \text{ and } I_{0}= \left\{ i \le n/2: |A_{i}| \ge fq_{0}  \right\}.\nonumber
    \end{align}
    Thourgh a simple estimating to obtain 
    \begin{align}
        |I_{0}| \ge md/9.\nonumber
    \end{align}
    With loss of generality we assume that $I_{0}=[|I_{0}|]$ and $N=\left \lceil  md/9 \right \rceil $. Then $[N] \subset I_{0}$.\par 
    For matrix $M \in E_{\omega}(x,y)$, we have 
    \begin{align}
        \Vert M(x+y)-\omega\Vert_{2}^{2}= \sum_{i=1}^{n/2}{|\langle \mathrm{R}_{i}(M),x+y  \rangle -\omega_{i}|^{2}} \le c_{\ref{Individual Probability}}md.\nonumber
    \end{align}
    Applying Markov's inequality for $a:=\min_{i \le k}|a_{i}|$, we have 
    \begin{align}
        \left|  \left\{  i \le N: \left| \langle \mathrm{R}_{i}(M), x+y  \rangle - \omega_{i} \right|   < a/3 \right\}  \right|  \ge md/9-9c_{\ref{Individual Probability}}md/a^{2} = N_{0},\nonumber
    \end{align}
    where $N_{0}=\left\lceil md/9-9c_{\ref{Individual Probability}}md/a^{2}  \right\rceil$.
    For $i \le N$, denote 
    \begin{align}
        \Omega_{i}=\left\{ M^{0} \in \mathcal{F}, M:\left| \langle \mathrm{R}_{i}(M), x+y  \rangle - \omega_{i} \right| \le a/3 \right\} \text{ and } \Omega_{0}=\mathcal{F}\left( \mathcal{I},V \right).\nonumber
    \end{align}
    Similar to the proof of the Lemma 6.11 in \cite{LT_duke}, we have
    \begin{align}
        \textsf{P}_{\mathcal{F}}(E_{\omega}(x,y)) \le \left( a^{2}e/(81c_{\ref{Individual Probability}}) \right)^{9c_{\ref{Individual Probability}}md/a^{2}} \prod_{i=1}^{N_{0}}\textsf{P}_{\mathcal{F}}\left( \Omega_{i} \right).\nonumber
    \end{align}
    Recall the definition of $\xi_{k}(i)= \xi_{k}(M,x+y,i)$ for $i \in I_{0}$ and $k \in A_{i}$. Then we have
    \begin{align}
        \textsf{P}_{\mathcal{F}}(\Omega_{i}) \le \mathcal{L}_{\mathcal{F}}\left( \sum_{k=0}^{q_{0}}{\xi_{k}(i)},a/3  \right) \le \mathcal{L}_{\mathcal{F}}\left( \sum_{k \in A_{i}}{\xi_{k}(i)}, a/3 \right) \le \frac{C\alpha}{\sqrt{(1-\alpha)fq_{0}}},\nonumber
    \end{align}
    where using the L\'{e}vy-Kolmogorov-Rogozin inequality in \cite{Tikhomirov} and set 
    $$\alpha:= \max_{k \in A_{i}}{\mathcal{L}_{\mathcal{F}}(\xi_{k}(i),a/3)}.$$
    Note that for $j_{k}=\mathrm{supp}\mathrm{R}_{i}(M^{0}) \cap J_{k}$ and $c=\xi_{ij_{k}}y_{1}$, we have
    \begin{align}
        \xi_{k}(i)=\sum_{j\in J_{k}}{(\delta_{ij}\xi_{ij})(x+y)}=\xi_{ij_{k}}x_{j_{k}}+c.\nonumber
    \end{align}
    If $j_{k} \in J_{k}^{(1)}$ we have $|\xi_{ij_{k}}x_{j_{k}}| \ge 2a/3$ and if $j_{k} \in J_{k}^{(2)} $ we have $x_{j_{k}}=0$. 
    \begin{align}
        \mathcal{L}_{\mathcal{F}}(\xi_{k}(i),a/3) = \mathcal{L}_{\mathcal{F}}\left( \xi_{ij_{k}}x_{j_{k}}+c,a/3 \right) \le 1/2:=\alpha.
    \end{align}
    Finally, similar to the proof of the lemma 6.11 in \cite{LT_duke}, we complete the proof of this lemma.
\end{proof}
We now give our main result in this subsection.
\begin{mypropo}\label{T2 and T3}
    Let $n \in \mathbb{N}^{+}$ be large enough and $p <c_{\ref{combinatorial lemma}}$ with $pn \ge 200\log{n}$. Then
    \begin{align}
        \textsf{P}\left(\exists x \in \mathcal{T}_{2} \cup \mathcal{T}_{3}:\Vert Mx\Vert_{2} \le  c_{\ref{T2 and T3}}\frac{(64p)^{\kappa}}{n^{2}(pn)^{3.5}}\Vert x\Vert_{2}  \right) \le \exp{(-10pn)},\nonumber
    \end{align}
    where $M$ from Theorem \ref{Theorem A} and $c_{\ref{T2 and T3}}>0 $ depending on $\xi$.
\end{mypropo}
\begin{proof}
    Fix $j \in \{2,3 \}$ and let 
    \begin{align}
        \mathcal{E}_{j}:=\left\{ \exists x\in \mathcal{T}_{j}:\Vert Mx\Vert_{2} \le \frac{\sqrt{c_{\ref{Individual Probability}}md}}{4b_{j}}\Vert x\Vert_{2}  \right\},\nonumber
    \end{align}
    where $b_{2}=C_{\ref{Steep Euclidean norm}}^{(2)}n^{2}(pn)^{3}/(64p)^{\kappa}$ and $b_{3}=C_{\ref{Steep Euclidean norm}}^{(2)}C_{\tau}n^{2}(pn)^{3.5}/(64p)^{\kappa}$.\par 
    For applying Lemma \ref{Spectral norm}, set $C=C_{\ref{Spectral norm}}(\xi,10)$ and 
    \begin{align}
        \mathcal{E}_{norm}:=\left\{ \Vert M -(pT+B)\textbf{1}\textbf{1}^{T}\Vert_{2} \le C\sqrt{pn} \right\}.\nonumber
    \end{align}
    Normalize $x\in \mathcal{T}_{j}$ so that $x \in \mathcal{T}_{j}^{*}$. Thus, let $\mathcal{N}_{j}=\mathcal{N}_{j}^{(1)}+\mathcal{N}_{j}^{(2)}$ be the net in Lemma \ref{Net for T2 and T3}. Then there exists $u \in \mathcal{N}_{j}^{(1)} $ such that $u_{n_{s_{0}+j-1}}^{*} \ge 2/3$ and $u_{l}^{*}=0$ for any $l > n_{s_{0}+j}$, and $v \in \mathcal{N}_{j}^{(2)} \subset \mathrm{Span}(\textbf{1})$ such that
    
    $$\Vert x-u-w\Vert_{\textbf{e}} \le \frac{\sqrt{2n}}{C_{\tau}\sqrt{d}}:=\varepsilon.$$
    
    Set $x-u-v=z+w$, where $w = P_{\textbf{e}^{\bot}}(x-u-v)$ and Let $M_{1}$ denote the random matrix consisting of the first $n/2$ rows of $M$ and $M_{2}$ denote the random matrix consisting of the last $n/2$ rows of $M$. Then conditioned on $\mathcal{E}_{j} \cap \mathcal{E}_{norm}$, we have
    \begin{equation}
        \begin{aligned}
            \Vert (M_{1}-M_{2})(u+v)\Vert_{2} 
            & \le 2\Vert Mx\Vert_{2}+ \Vert(M_{1}-M_{2})(x-u-v)\Vert_{2}\\
            & \le \sqrt{c_{\ref{Individual Probability}}md}/2+\Vert (M_{1}-M_{2})(z+w)\Vert_{2}\\
            & \le 2\Vert (M- \textsf{E}M )z\Vert_{2}+2\Vert (M-\textsf{E}M)\textbf{1}\Vert_{2}\frac{\varepsilon}{\sqrt{p}n}+\sqrt{c_{\ref{Individual Probability}}md}/2\\
            & \le \sqrt{c_{\ref{Individual Probability}}md},\nonumber
        \end{aligned}
    \end{equation}
    where using $\Vert w\Vert_{2} \le \frac{\varepsilon}{\sqrt{p}n}$.\par
    Next, we using Lemma \ref{Individual Probability} for $m_{1}=m_{0}=n_{s_{0}+1}$, $m_{2}=n_{s_{0}+2}$ or $m_{1}=n_{s_{0}+3}$ and $m_{1}=n_{s_{0}+2}>m_{0}=n_{s_{0}+1}$, we have 
    \begin{align}
        \textsf{P}\left( \cup_{j=2,3}\mathcal{E}_{j}\cap \mathcal{E}_{norm} \cap \mathcal{E}_{card} \right) \le \exp(-10pn),\nonumber
    \end{align}
    where $\mathcal{E}_{card}$ is the event introduced in Lemma \ref{Individual Probability} and we fix the $M_{2}$ or $M_{1}$.\par 
    We now complete the proof of this poposition by Lemma \ref{Spectral norm} and \ref{combinatorial lemma of T2 and T3}.
\end{proof}

\subsection{\texorpdfstring{$\mathcal{R}$}{R}-vectors}\label{R-vectors section}
In this subsection, we introduce the following bound of $\mathcal{R}$-vectors.
\begin{mypropo}\label{R-vector}
    There are absolute constants $r_{0}$, $\rho_{0}$, constants $C_{\ref{R-vector}}$ and $C'_{\ref{R-vector}}$ depending only on $\xi$ such that the following holds. Let $r \in (0,r_{0})$, $\rho \in (0,\rho_{0})$, $n \in \mathbb{N}^{+}$, and $p \in (0, 0.001]$ be such that $d=pn  \ge C_{\ref{R-vector}}\log{n}$. Then, we have
    \begin{align}
        \textsf{P}\left( \exists x \in \mathcal{R}:\Vert Mx\Vert_{2} \le C'_{\ref{R-vector}}\sqrt{p}n  \right) \le \exp{(-100pn)}.\nonumber
    \end{align}
\end{mypropo}
\begin{proof}
    We provide a concise proof. To establish the aforementioned estimate, we first define $M_{1}$ and $M_{2}$ analogously to the proof of Proposition \ref{T2 and T3}. Then, considering $\Vert (M_{1}-M_{2})(x)\Vert_{2}$ for $x \in \mathcal{R}$, we apply a net argument similar to Lemma 6.8 in \cite{LT_duke}(or Lemma \ref{Net for T2 and T3} in this paper), and finally combine Lemmas \ref{Spectral norm} and \ref{Levy concentration} to complete the proof of this proposition.
\end{proof}
\section{Proof of main results}\label{Proof of main result section}
The main goal of this section is to prove Theorem \ref{Theorem A}, combining the results of Sections \ref{Preliminaries}, \ref{Unstructured vectors section}, and \ref{Structured vectors section}. Our first step is to show that, with high probability, any vector orthogonal to $M_{n}$ is unstructured.
\begin{mycoro}\label{Structured part}
    There exist $C_{\ref{Structured part}}>1>c_{\ref{Structured part}}$, $\delta,\rho \in (0,1)$ and $r \in (0,1)$ depending on $\xi$ such that the following holds. Let $M_{n}$ be an $n \times n$ random matrix from Theorem \ref{Theorem A} with $n \ge C_{\ref{Structured part}}$ and let $\textbf{g}(\cdot)$ is a growth function satisfying \eqref{g in structured vectors}. Then 
    \begin{equation}
    \begin{aligned}
        \textsf{P} & \bigg( \Vert M_{n}x\Vert_{2}
         \le a_{n}^{-1}\Vert x\Vert_{2} \text{ for some } x \notin \bigcup_{\lambda \ge 0}\left( \mathcal{V}_{n}(r,\textbf{g},\delta,\rho) \right)  \bigg)\\
         & =(1+o_{n}(1))n\textsf{P}(\eta=0)^{n}+(1+o_{n}(1))\binom{n}{2}\textsf{P}(\eta'=\eta)^{n},\nonumber
        \end{aligned}
    \end{equation}
    where $\eta'$ is a independent copy of $\eta$ and 
    \begin{align}
        a_{n}:=\frac{n^{2}(pn)^{2}}{c_{\ref{Structured part}}(64p)^{\kappa}}\min(1,p^{1.5}n), \ \text{ and } \kappa:=\frac{\log(\gamma pn)}{\log(\left\lfloor pn/(4\log(1/p)) \right\rfloor )}.\nonumber
    \end{align}
\end{mycoro}
\begin{proof}
    In fact, the proof follows directly by combining the results of Section \ref{Structured vectors section} with Proposition \ref{Structured vectors partition}.
\end{proof}
For the unstructured vector, as a version of classic ``invertibility via distance ", we have the following lemma.
\begin{mylem}\label{Invertibility via distance}
    Let $r,\delta,\rho \in (0,1)$ and $\textbf{g}$ be a growth function. Let $n \ge 30/r$ and $M_{n}$ be an $n \times n $ random matrix from Theorem \ref{Theorem A}. Then for all $\varepsilon>0$ we have
    \begin{equation}
        \begin{aligned}
           \textsf{P}\big( & \Vert M_{n}x\Vert_{2} \le \varepsilon \Vert x\Vert_{2} \text{ for some } x \in \mathcal{V}_{n}(r,\textbf{g},\delta,\rho)  \big)\\
           & \le \frac{8}{(rn)^{4}}\sum_{i_{1},\dots,i_{4}}\textsf{P}\left( \mathrm{dist}\left( \mathrm{C}_{i_{j}}(M_{n}),H_{i_{j}}(M_{n})\right) \le \varepsilon
           b_{n}  \text{ for all } j \in [4]\right),\nonumber
        \end{aligned}
    \end{equation}
    where the summand is taken from all pairs $(i_{1},\dots,i_{4})$ with $i_{j} \ne i_{t}$ for $j \ne t$ and $b_{n}=\sum_{i=1}^{n}\textbf{g}(i)$.
\end{mylem}
With all requisite groundwork now firmly in place, we proceed to the proof of the main theorem.
\begin{proof}[\textsf{Proof of Theorem \ref{Theorem A}}]
    Fix parameters $r,\delta,\rho,a_{n}$ be taken in Corollary \ref{Structured part}, $\textbf{g}(\cdot)$ satisfying \eqref{g in structured vectors}, and $b_{n}:=\sum_{i=1}^{n}\textbf{g}(i)$. Denote by $\mathcal{E}$ the complement of the event
    \begin{align}
        \left\{ \Vert M_{n}x\Vert_{2}
         \le a_{n}^{-1}\Vert x\Vert_{2} \text{ or } \Vert M_{n}^{T}x\Vert_{2}
         \le a_{n}^{-1}\Vert x\Vert_{2} \text{ for some } x \notin \bigcup_{\lambda \ge 0}\left( \mathcal{V}_{n}(r,\textbf{g},\delta,\rho) \right) \right\}.\nonumber
    \end{align}
    For $i\in [4]$ denote 
    \begin{align}
        \mathcal{E}_{i}:=\left\{  \mathrm{dist}\left( \mathrm{C}_{i}(M_{n}),H_{i}(M_{n}) \right) \le a_{n}^{-1}t \right\}.\nonumber
    \end{align}
    Applying Corollary \ref{Structured part} and Lemma \ref{Invertibility via distance}, we have 
    \begin{equation}
        \begin{aligned}
            \textsf{P}\left( s_{\min}(M_{n}) \le t(a_{n}b_{n})^{-1} \right) \le P_{s}+\frac{8}{r^{4}}\textsf{P}\left(\mathcal{E}\cap \bigcap_{i=1}^{4}\mathcal{E}_{i}  \right),\nonumber
        \end{aligned}
    \end{equation}
    where 
    \begin{align}
        P_{s}:=(2+o_{n}(1))n\textsf{P}(\eta=0)^{n}+(1+o_{n}(1))n(n-1)\textsf{P}(\eta'=\eta)^{n}.\nonumber
    \end{align}
    Consider the events for $i \le 4$,
    \begin{align}
        \Omega_{i}:=\left\{  \left| \left\{ j \in [n]:\eta_{ij}=b  \right\} \right| \in [pn/8,8pn] \right\} \text{ and } \Omega:=\bigcup_{i=1}^{4}\Omega_{i}.\nonumber
    \end{align}
    Applying Lemma \ref{pn/8-8pn}, we obtain
    \begin{align}
        \textsf{P}\left( \Omega^{c} \right) \le (1-p)^{2n}.\nonumber
    \end{align}
    Furthermore, we get 
    \begin{align}
        \textsf{P}\left( \mathcal{E}\cap\bigcap_{i=1}^{4}\mathcal{E}_{i} \right) \le (1-p)^{2n}+\textsf{P}\left( \mathcal{E}\cap \mathcal{E}_{1}\cap\Omega_{1}  \right).\nonumber
    \end{align}
    Let $\textbf{Y}$ be a random unit vector orthogonal to $H_{1}(M_{n})$, consider on $\mathcal{E}_{1}$,
    \begin{align}
        \left| \left\langle \textbf{Y},\mathrm{C}_{i}(M_{n})  \right\rangle  \right| \le \Vert M_{n}^{T}\textbf{Y}\Vert_{2} \le a_{n}^{-1}t.\nonumber
    \end{align}
    It implies that $\textbf{Z}:=\textbf{Y}/\textbf{Y}_{\left\lfloor rn \right\rfloor}^{*} \in \mathcal{V}_{n}(r,\textbf{g},\delta,\rho)$. Furthermore, we have 
    \begin{align}
        P_{0}=\textsf{P}\left( \mathcal{E}_{1}\cap \mathcal{E}\cap\Omega  \right) \le \textsf{P}\left( \left\{ \exists \textbf{Z}\in H_{1}^{\perp}\cap\mathcal{V}_{n}:\Vert M_{n} \textbf{Z}\Vert_{2} \le t a_{n}^{-1}b_{n} \right\} \cap \Omega_{1}   \right).\nonumber
    \end{align}
    Applying Theorem \ref{Large U degree} for $R =4$, there are $K_{1} \ge 1$ and $K_{2} \ge4$ such that with probability at least $1-e^{-4pn}$,
    \begin{equation}
        \begin{aligned}
            H_{1}(M_{n})^{\perp} &  \cap \mathcal{V}_{n}(r,\textbf{g},\delta,\rho)\\
             \subset & \big\{ x \in \mathbb{R}^{n}:x_{\left\lfloor rn \right\rfloor}^{*}=1,\mathrm{UD}_{n}^{\xi}(x,m,K_{1},K_{2}) \ge \exp(Rpn) \\
            & \text{ for all } pn/8 \le m \le 8pn \big\}.\nonumber
        \end{aligned}
    \end{equation}
    Thus, we get 
    \begin{equation}
        \begin{aligned}
        P_{0} 
        & \le \exp(-4pn)+\sup_{\substack{m \in [pn/8,8pn],y \in \Upsilon(r) \\ \mathrm{UD}_{n}^{\xi}(y,m,K_{1},K_{2}) \ge \exp(4pn)}}\textsf{P}\left(\left| \left\langle y,\mathrm{C}_{i}(M_{n})  \right\rangle  \right|\le t a_{n}^{-1}b_{n} \big| \Omega_{1}  \right)\\
        & \le (1+C_{\ref{Small ball prob via UD}})\exp(-4pn)+\frac{C_{\ref{Small ball prob via UD}}b_{n}}{a_{n}\sqrt{pn/8}}t .\nonumber
    \end{aligned}
    \end{equation}
    Therefore,
    \begin{align}
        \textsf{P}\left( s_{\min}(M_{n}) \le t(a_{n}b_{n})^{-1} \right) \le P_{s}+\frac{Cb_{n}}{a_{n}r^{4}\sqrt{pn}}t.\nonumber
    \end{align}
    As the last step, by rescaling $t$ we have 
    \begin{align}
        \textsf{P}\left( s_{\min}(M_{n}) \le t\frac{ cr^{2}\sqrt{pn} }{b_{n}^{2}} \right) \le P_{s}+t.\nonumber
    \end{align}
    Note that for large $n$, we have 
    \begin{align}
        \frac{ cr^{2}\sqrt{pn} }{b_{n}^{2}} \ge \exp(-3\log^{2}(2n)),\nonumber
    \end{align}
    which implies the result.
\end{proof}
\textbf{Acknowledgment:} I would like to thank Prof. Hanchao Wang for his valuable remarks and support. This work was supported by the Youth Student Fundamental study Funds of Shandong University  (No. SDU-QM-B202407).

%%===========================================================================================%%
%% If you are submitting to one of the Nature Portfolio journals, using the eJP submission   %%
%% system, please include the references within the manuscript file itself. You may do this  %%
%% by copying the reference list from your .bbl file, paste it into the main manuscript .tex %%
%% file, and delete the associated \verb+\bibliography+ commands.                            %%
%%===========================================================================================%%

\end{document}